\documentclass{amsart}

\usepackage[all]{xy}
\usepackage{amsmath}
\usepackage{amssymb}
\usepackage{amsthm}
\usepackage{latexsym}
\usepackage{enumerate}
\usepackage{prettyref}
\usepackage{hyperref}

\newcommand{\Ao}{\ensuremath{\mathbb{A}^1}}

\newcommand{\topos}[1]{\ensuremath{\mathcal{S}hv(#1)}}
\newcommand{\simplicial}[1]{\ensuremath{\Delta^{op}#1}}
\newcommand{\sms}{\ensuremath{\operatorname{Sm}_S}}
\newcommand{\pt}{\ensuremath{\ast}}

\newcommand{\Hom}{\ensuremath{\operatorname{Hom}}}
\newcommand{\set}{\ensuremath{\mathcal{S}et}}
\newcommand{\holim}{\ensuremath{\operatornamewithlimits{holim}}}
\newcommand{\id}{\ensuremath{\operatorname{id}}}

\newcommand{\haut}{\ensuremath{\operatorname{hAut}_\bullet}}

\newcommand{\diagram}[1]{\ensuremath{\mathcal{#1}}}
\newcommand{\category}[1]{\ensuremath{\mathcal{#1}}}
\newcommand{\colim}{\ensuremath{\operatornamewithlimits{colim}}}
\newcommand{\hocolim}{\ensuremath{\operatornamewithlimits{hocolim}}}
\newcommand{\inthom}{\ensuremath{\mathbf{Hom}}}
\newcommand{\hofib}{\ensuremath{\operatorname{hofib}}}

\newtheorem{mainthm}{Theorem}
\newtheorem{theorem}{Theorem}[section]
\newrefformat{thm}{\hyperref[{#1}]{Theorem~\ref*{#1}}}
\newtheorem{definition}[theorem]{Definition}
\newrefformat{def}{\hyperref[{#1}]{Definition~\ref*{#1}}}
\newtheorem{lemma}[theorem]{Lemma}
\newrefformat{lem}{\hyperref[{#1}]{Lemma~\ref*{#1}}}
\newtheorem{proposition}[theorem]{Proposition}
\newrefformat{prop}{\hyperref[{#1}]{Proposition~\ref*{#1}}}
\newtheorem{corollary}[theorem]{Corollary}
\newrefformat{cor}{\hyperref[{#1}]{Corollary~\ref*{#1}}}
\newtheorem{remark}[theorem]{Remark}
\newrefformat{rem}{\hyperref[{#1}]{Remark~\ref*{#1}}}
{
\newtheorem{examplecore}[theorem]{Example}}
\newrefformat{ex}{\hyperref[{#1}]{Example~\ref*{#1}}}

\newenvironment{example}{\begin{examplecore}}{\hspace*{\fill}
$\square$\par\vspace{.1cm}\end{examplecore}}

\begin{document}

\title{Classifying Spaces and Fibrations of Simplicial Sheaves}

\author{Matthias Wendt}
\email{matthias.wendt@math.uni-freiburg.de}
\address{Mathematisches Institut\\
Albert-Ludwigs-
Uni\-ver\-si\-t\"at Freiburg\\ 
Eckerstra\ss{}e 1\\ 
79104, Freiburg im Breisgau\\ 
Germany}


\begin{abstract}
In this paper, we discuss the construction of
classifying spaces of fibre sequences in model categories of
simplicial sheaves. One construction proceeds via Brown
representability and provides a classification in the pointed model
category. The second construction is given by the classifying 
space of the monoid of homotopy self-equivalences of a
simplicial sheaf and provides the unpointed classification.
\end{abstract}

\maketitle
\setcounter{tocdepth}{1}
\tableofcontents

\section{Introduction}

In this paper, we discuss the classification of fibrations in
categories of simplicial sheaves. As usual, the results are
modelled on the corresponding results for simplicial sets or
topological spaces which we first discuss. 

For simplicial sets, there are two approaches to the construction of
classifying spaces. The first approach uses Brown representability to
classify rooted fibrations, yielding a classification in 
the pointed category. This line of construction has been pursued
in the work of Allaud \cite{allaud:1966:fibre}, Dold \cite{dold:1966},
and Sch\"on \cite{schoen:1982}.
The second approach applies the bar construction to the monoid of
homotopy self-equivalences of the fibre. This is a generalization of
the classifying space of a topological group, cf. \cite{milnor:1956},
further developed by Dold and Lashof for associative $H$-spaces,
cf. \cite{dold:lashof}, and applied to the classification of
fibrations in \cite{stasheff:1963} and
\cite{may:1975:fibrations}. This approach yields a 
classification in the unpointed category. 
The two approaches do not yield equivalent classifying spaces: the
rooted fibrations carry an action of the group of homotopy
self-equivalences, and dividing out this action yields the unpointed
classifying space of the second approach. A survey on construction of
classifying spaces and classification of fibrations can be found in
\cite{stasheff:1970} or \cite{may:1975:fibrations}.

Now we want to explain why  this theory works in the general setting
of simplicial sheaves. On the one hand, \emph{fibrations of simplicial
  sheaves 
  can be glued}. This is of course not true on the nose, but as for
simplicial sets there is a way around this problem. The essence of the
solution is that some kind of ``homotopy distributivity'' holds -- in
some situations it is possible to interchange homotopy limits and
homotopy colimits. The notion of homotopy distributivity is due to
Rezk \cite{rezk:1998:sharp} and can be used to generalize various
classical results on homotopy pullbacks and homotopy colimits, such as
Puppe's theorem or Mather's cube theorem. This theory is developed in
\prettyref{sec:hocolim}. Once such a glueing for fibrations of
simplicial sheaves is developed, it is a simple matter to prove
that the conditions for a version of Brown representability are
satisfied, yielding classifying spaces for analogs of rooted
fibrations of simplicial sheaves. 
On the other hand, \emph{fibrations of simplicial sheaves
  correspond to principal bundles under homotopy
  self-equivalences}. Suitably formulated, we can associate to a
simplicial sheaf $X$ a simplicial sheaf of monoids consisting of
homotopy self-equivalences of $X$. To this monoid we can apply the
bar construction. One can prove that the resulting space classifies
fibre sequences of simplicial sheaves. 

In our approach to the construction of classifying spaces, we
introduce a notion of local triviality of fibrations in the
Grothendieck topology. This condition is one possible generalization of
the usual condition that all fibres of the fibration should have the
homotopy type of the given fibre $F$. In the first approach via Brown
representability, this condition ensures that the fibre functor is
indeed set-valued. In the second approach using the bar construction,
it comes in naturally because we can not talk about fibre sequence if
the base is not pointed.

The main result of the paper is the following theorem:
\begin{mainthm}
Let $T$ be a site.
\begin{enumerate}[(i)]
\item Assume that the category $\simplicial{\topos{T}}$ of simplicial
  sheaves on $T$ is compactly
  generated, cf. \prettyref{def:cptjardine}. 
Let $F$ be a pointed simplicial sheaf on $T$. There 
exists a pointed simplicial sheaf $B^fF$ which classifies locally
trivial fibrations with fibre $F$ up to (rooted) equivalence, i.e. for
each pointed simplicial sheaf $X$ there is a bijection between the set
of equivalence classes of fibre sequences over $X$ with fibre
$F$ and the set of pointed homotopy  classes of maps $X\rightarrow BF$. 
\item Let $F$ be a simplicial sheaf on $T$. There exists a simplicial sheaf
  denoted by   $B(\ast,\haut(F),\ast)$ which classifies locally
  trivial morphisms   with fibre $F$ up to equivalence, i.e. for each
  simplicial sheaf   there is a bijection between the set of
  equivalence   classes of   locally trivial morphisms over $X$ with
  fibre $F$ and the set of   unpointed homotopy classes of maps
  $X\rightarrow   B(\ast,\haut(F),\ast)$. 
\end{enumerate}
\end{mainthm}
The two classification results can be found in
\prettyref{thm:classify} and \prettyref{thm:classmay}. The main input
in both of them is homotopy distributivity which originally is a
result of Rezk \cite{rezk:1998:sharp}. We give a short proof for topoi
with enough points in \prettyref{prop:hodistrib}. 

One word on the relation between our approach and the classification
results in \cite{dwyer:kan:1984}: given fixed simplicial sheaves $B$
and $F$, analogs of the classification results of
\cite{dwyer:kan:1984} can be used to construct a simplicial set whose
components are in one-to-one correspondence with fibre sequences over
$B$ with fibre $F$. However, these results do not imply that the
various simplicial sets are the sections of one simplicial sheaf. It
is exactly this \emph{internal} classification that we are after. For
this, some sort of homotopy distributivity is needed, as we discuss in
\prettyref{sec:hocolim}.  

Finally, a short sketch of the envisioned applications is in
order. The main motivation for the research reported in this paper
comes from $\Ao$-homotopy theory, which is a homotopy theory for
algebraic varieties defined by Morel and Voevodsky
\cite{morel:voevodsky:1999:a1}. On the one hand, the homotopy
distributivity results from \prettyref{sec:hocolim} have been used in
\cite{toric} to give descriptions of $\Ao$-fundamental groups of
smooth toric varieties. On the other hand, the theory of classifying
spaces developed here allows several results on unstable localization
of fibre sequences for simplicial sets to be carried over to
simplicial sheaves. This is discussed in \cite[Chapter 4]{thesis}
and will be further elaborated in a forthcoming paper. 
The most interesting application, however, is in $\Ao$-homotopy theory. 
The results presented here allow the construction of classifying
spaces, and the localization theory of \cite{thesis} allows us to  obtain
checkable conditions under which  fibrations which are 
locally trivial in the Nisnevich topology are indeed $\Ao$-local. This
will be discussed in \cite{flocal}.

\emph{Structure of the Paper:}
In \prettyref{sec:hocolim}, we develop the necessary preliminaries for
homotopy distributivity which will be needed. \prettyref{sec:fibseq}
we discuss locally trivial fibrations in categories of simplicial
sheaves. Then the two classification results are proved in
\prettyref{sec:one} and \prettyref{sec:two}.

\emph{Acknowledgements:} 
The results presented here are taken from my PhD thesis \cite{thesis}
which was  supervised by Annette Huber-Klawitter. I would like to use
the opportunity to thank her for her encouragement and interest in my 
work. I would also like to thank the following people: Fernando Muro
and Ji\v r\'i Rosick\'y for explanations on Brown representability,
the referee for pointing out some mistakes concerning pointed
vs. unpointed classification, and Paul Goerss, Jim Stasheff and an
editor of JHRS for helpful comments. 

\section{Homotopy Limits and Colimits of Simplicial Sheaves}
\label{sec:hocolim}

\subsection{Model Structures for Simplicial Sheaves}

The global pattern in the theory of
model structures on categories of simplicial sheaves is always the
same: a category of simplicial sheaves behaves in many aspects like 
the category of simplicial sets. This is also evident in the proofs,
which reduce statements about simplicial sheaves to known statements
about simplicial sets. 

The basic definitions of sites and categories of sheaves on them can
be found in \cite{maclane:moerdijk:1992:sheaves}. 
We will freely use these as well as the notions of homotopical
algebra. For the definition of model categories, see
\cite{goerss:jardine:1999:simplicial} with a particular focus on
simplicial sets, as well as \cite{hovey:1998:modelcats} and
\cite{hirschhorn:2003:modelcats}. 

We denote by
$\simplicial{\mathcal{C}}$ the category of simplicial objects in the
category $\mathcal{C}$. In particular, the category of simplicial
sheaves on a site $T$ is denoted by $\simplicial{\topos{T}}$.

The following comprises the main facts about model structures on
simplicial sheaves. 

\begin{theorem}
\label{thm:jardine}
Let $\mathcal{E}$ be a topos. Then the category
$\simplicial{\mathcal{E}}$ of simplicial objects in $\mathcal{E}$ has
a model structure, where the
\begin{enumerate}[(i)]
\item cofibrations are monomorphisms,
\item weak equivalences are detected on a fixed Boolean localization,
\item fibrations are determined by the right lifting property. 
\end{enumerate}

The above definition of weak equivalences does not depend on
the Boolean localization. 

The injective model
structure of Jardine on the category of (pre-)sheaves of simplicial
sets on $T$ is a proper simplicial and cellular model
structure. 
\end{theorem}

Existence is proved in  \cite[Theorems 18 and
27]{jardine:1996:boolean}. Properness and simpliciality
are proven in \cite[Theorem 24]{jardine:1996:boolean}. The fact that
the model categories are cofibrantly generated is implicit in
Jardine's proofs, though not explicitly stated. 
The combinatoriality follows since categories of sheaves on a
Grothendieck site are locally presentable.
Cellularity is proven in \cite[Theorem 1.4]{hornbostel:2006}.

\subsection{Recollection on Homotopy Limits and Colimits}

Homotopy colimits and limits are homotopy-invariant
versions of the ordinary colimits and limits for
categories. Abstractly, one can define the ordinary colimit of a
diagram $\diagram{X}:\category{I}\rightarrow \category{C}$ as left
adjoint of the diagonal functor
$\Delta_{\category{I}}:\category{C}\rightarrow
\hom(\category{I},\category{C})$, where
$\hom(\category{I},\category{C})$ is the category of
$\category{I}$-diagrams in $\category{C}$. Similarly, the ordinary
limit is the right adjoint of the diagonal, cf. \cite[Section
  X.1]{maclane:1998:categories}. Homotopy colimits and limits are then
  defined as suitable derived functors of the ordinary colimit and
  limit functors.

A general reference for homotopy limits and colimits is
\cite{hirschhorn:2003:modelcats}, in the context of simplicial
sheaves see also \cite{morel:voevodsky:1999:a1}.  
We shortly recall the definition of homotopy limits and colimits. 

\begin{definition}
Let $\mathcal{C}$ be a cofibrantly generated simplicial model
category, and $\mathcal{I}$ be any small category. 
\begin{description}
\item[Colimits] The category $\Hom(\mathcal{I},\mathcal{C})$ of
  $\mathcal{I}$-indexed diagrams in $\mathcal{C}$ has the structure of
  a simplicial model category by taking the weak equivalences and
  fibrations to be the pointwise ones. Then the diagonal
$\Delta:\mathcal{C}\rightarrow\Hom(\mathcal{I},\mathcal{C})$ preserves
fibrations and weak equivalences, and therefore is a right Quillen
functor. Its left adjoint
$\colim:\Hom(\mathcal{I},\mathcal{C})\rightarrow \mathcal{C}$ is thus
a left Quillen functor, and we can define its derived functor
\begin{displaymath}
\hocolim_{\mathcal{I}}=L\colim_{\mathcal{I}}:\mathcal{X}\mapsto
\colim_{\mathcal{I}}Q\mathcal{X}, 
\end{displaymath}
where $Q$ is a cofibrant replacement in the model category
$\Hom(\mathcal{I},\mathcal{C})$. 
\item[Limits] Dually, the category $\Hom(\mathcal{I},\mathcal{C})$
  also has a simplicial model structure  where the weak
  equivalences and cofibrations are the pointwise ones. Then the diagonal
$\Delta:\mathcal{C}\rightarrow\Hom(\mathcal{I},\mathcal{C})$ preserves
cofibrations and weak equivalences, and therefore is a left Quillen 
functor. Its right adjoint
$\lim:\Hom(\mathcal{I},\mathcal{C})\rightarrow \mathcal{C}$ is thus
a right Quillen functor, and we can define its derived functor
\begin{displaymath}
\holim_{\mathcal{I}}=R\lim_{\mathcal{I}}:\mathcal{X}\mapsto
\lim_{\mathcal{I}}R\mathcal{X}, 
\end{displaymath}
where $R$ is a fibrant replacement in the model category
$\Hom(\mathcal{I},\mathcal{C})$. 
\end{description}
We usually denote the homotopy colimit of an $\mathcal{I}$-diagram
$\mathcal{X}$  by $\hocolim_{\category{I}}\mathcal{X}$, the special
case of a homotopy pushout is denoted by  $A\cup_B^hC$. 
Similarly, homotopy limits are usually denoted by
$\holim_{\category{I}}\mathcal{X}$, and the homotopy pullbacks by
$A\times_B^hC$.  
\end{definition}

There are also more concrete constructions of homotopy limits and
colimits. Since we are not going to need these descriptions, we just
refer to \cite[Chapter 18]{hirschhorn:2003:modelcats}.  

The fact that homotopy colimits resp. limits can be defined as left
resp. right derived functors of colimits resp. limits implies that
they are homotopy invariant \cite[Theorem
  18.5.3]{hirschhorn:2003:modelcats}.  

\begin{proposition}
\label{prop:hoinvariance}
Let $\category{C}$ be a simplicial model category, and let
$\category{I}$ be a small category. If $f:\diagram{X}\rightarrow
\diagram{Y}$ is a morphism of $\category{I}$-diagrams of cofibrant
objects in $\category{C}$ which is an objectwise equivalence, then  
\begin{displaymath}
\hocolim_{\category{I}}
f:\hocolim_{\category{I}}\diagram{X}\rightarrow
\hocolim_{\category{I}}\diagram{Y}
\end{displaymath}
is a weak equivalence of cofibrant objects. 

Dually, if $f:\diagram{X}\rightarrow \diagram{Y}$ is a morphism of
$\category{I}$-diagrams in $\category{C}$ which is an objectwise
equivalence of fibrant objects, then 
\begin{displaymath}
\holim_{\category{I}}
f:\holim_{\category{I}}\diagram{X}\rightarrow
\holim_{\category{I}}\diagram{Y}
\end{displaymath}
is a weak equivalence of fibrant objects.
\end{proposition}

Moreover, homotopy colimits and limits interact nicely with the
corresponding left resp. right Quillen functors. 
\begin{proposition}
\label{prop:commute}
Let $F:\mathcal{C}\rightarrow \mathcal{D}$ be a left Quillen
functor. Then the following diagram commutes up to isomorphism:
\begin{center}
  \begin{minipage}[c]{10cm}
    \xymatrix{
      \operatorname{Ho}\Hom(\mathcal{I},\mathcal{C}) \ar[rr]^\hocolim
      \ar[d]_{LF} && \operatorname{Ho}\mathcal{C} \ar[d]^{LF} \\
      \operatorname{Ho}\Hom(\mathcal{I},\mathcal{D}) \ar[rr]_\hocolim &&
      \operatorname{Ho}\mathcal{D}, 
    }
  \end{minipage}
\end{center}
\end{proposition}
One example of this situation is the relation between 
homotopy colimits and hom-functors as stated in \cite[Lemma
  2.1.19]{morel:voevodsky:1999:a1}.

Finally, we state a standard fact on homotopy pullbacks, cf. also
\cite[Lemma II.8.22]{goerss:jardine:1999:simplicial}:  

\begin{lemma}
\label{lem:hopullback}
Let $\mathcal{C}$ be a proper model category, and let the following
commutative diagram be given:
\begin{center}
  \begin{minipage}[c]{10cm}
    \xymatrix{
      X_1\ar[r]\ar[d] & X_2 \ar[r]\ar[d] & X_3\ar[d] \\
      Y_1\ar[r] & Y_2 \ar[r] & Y_3.
    }
  \end{minipage}
\end{center}
If the inner squares are homotopy pullback squares, then so is the
outer. If the outer square and the right inner square are homotopy
pullback squares, then so is the left inner square.
\end{lemma}

\subsection{Functorialities} 

We first recall the basic result that geometric morphisms of
Grothendieck topoi induce Quillen functors. This is basically a
reformulation of  \cite[Proposition 2.1.47]{morel:voevodsky:1999:a1}.  

\begin{proposition}
\label{prop:pair}
Let $f:\mathcal{F}\rightarrow \mathcal{E}$ be a geometric morphism of
Grothendieck topoi. We also denote by
$f^\ast:\simplicial{\mathcal{E}}\rightarrow \simplicial{\mathcal{F}}$
and $f_\ast:\simplicial{\mathcal{F}}\rightarrow
\simplicial{\mathcal{E}}$ the induced functors on the categories of
simplicial sheaves. 
Then $(f^\ast,f_\ast)$ is a Quillen pair, i.e. $f^\ast$ preserves
cofibrations and trivial cofibrations and $f_\ast$ preserves
fibrations and trivial fibrations.
\end{proposition}

Finally, we recall that weak equivalences are reflected along 
surjective geometric morphisms.

\begin{proposition}
\label{prop:reflect}
Let $f:\mathcal{E}'\rightarrow \mathcal{E}$ be a surjective
geometric morphism, and let $g:A\rightarrow B$ be a morphism in
$\mathcal{E}$. Then $g$ is a weak equivalence if $f^\ast g:f^\ast
A\rightarrow f^\ast B$ is a weak equivalence in $\mathcal{E}'$.
\end{proposition}

\begin{proof}
If $f$ is surjective, then any Boolean localization of
$\mathcal{E}'$ is a Boolean localization of $\mathcal{E}$, because a
Boolean localization of $\mathcal{E}$ is simply a surjective geometric 
morphism $\mathcal{B}\rightarrow \mathcal{E}$, where $\mathcal{B}$ is
the topos of sheaves on a complete Boolean algebra. In
\cite{jardine:1996:boolean}, it was proved that the weak equivalences
which are defined via Boolean localizations are independent of the
Boolean localization. 

A morphism $f:A\rightarrow B$ is thus a weak equivalence in
$\mathcal{E}$ if it is a morphism after pullback along
$f^\ast:\mathcal{E}\stackrel{g^\ast}{\rightarrow}
\mathcal{E}'\rightarrow\mathcal{B}$, 
where the latter morphism is a chosen Boolean localization of
$\mathcal{E}'$. But by definition, this is equivalent to the fact that
$g^\ast f$ is a weak equivalence in $\mathcal{E}'$. This proves the
claim. 
\end{proof}

\subsection{Homotopy Colimits} 

In this subsection, we recall the behaviour of homotopy colimits under
the inverse image part of a geometric morphism. The inverse image
preserves homotopy colimits, and reflects them if the geometric
morphism is surjective. 

\begin{proposition}
\label{prop:colim1}
Let $\mathcal{E}$ be a topos, and let $f:\mathcal{E}'\rightarrow
\mathcal{E}$ be a geometric morphism. Then
$f^\ast:\simplicial{\mathcal{E}}\rightarrow \simplicial{\mathcal{E}'}$
preserves homotopy colimits.  
\end{proposition}

\begin{proof}
$f^\ast$ is a left Quillen functor, cf. Proposition \ref{prop:pair}.
The result follows from Proposition \ref{prop:commute}.
\end{proof}

\begin{proposition}
\label{prop:colim2}
Let $\mathcal{E}$ be a topos, let $\mathcal{I}$ be a small category,
and let $f:\mathcal{E}'\rightarrow \mathcal{E}$ be a geometric
morphism. If $f$ is surjective, then 
$f^\ast:\simplicial{\mathcal{E}}\rightarrow \simplicial{\mathcal{E}'}$ 
reflects homotopy colimits. In other words,
$\mathcal{X}:\mathcal{I}\rightarrow \simplicial{\mathcal{E}}$ is a
homotopy colimit diagram if and only if
$f^\ast\mathcal{X}:\mathcal{I}\rightarrow \simplicial{\mathcal{E}'}$
is a homotopy colimit diagram. 
\end{proposition}

\begin{proof}
Recall that $\mathcal{X}$ is a homotopy colimit diagram if the natural
map 
\begin{displaymath}
\Psi:\hocolim_{\mathcal{I}}\mathcal{X}\rightarrow
\colim_{\mathcal{I}}\mathcal{X}
\end{displaymath}
is a weak equivalence. 

We have a diagram
\begin{center}
  \begin{minipage}[c]{10cm}
    \xymatrix{
      f^\ast \hocolim \mathcal{X} \ar[r] & f^\ast\colim
      \mathcal{X}  \\
      \hocolim f^\ast\mathcal{X} \ar@{.>}[u]\ar[r] & \colim f^\ast \mathcal{X}
      \ar[u] .
    }
  \end{minipage}
\end{center}
The left arrow exists because to compute $\hocolim\mathcal{X}$, we use
a cofibrant replacement which is preserved by the left Quillen functor
$f^\ast$. Therefore there is a cone from the cofibrant diagram
$\mathcal{X}$ to $f^\ast\hocolim\mathcal{X}$ which has to factor
through the colimit, which is also the homotopy colimit since the
diagram is cofibrant. The vertical morphisms are weak equivalences by
Proposition \ref{prop:commute}, hence $f^\ast\Psi$ can be identified up to
weak equivalence with the map 
\begin{displaymath}
\hocolim_{\mathcal{I}}f^\ast\mathcal{X}\rightarrow
\colim_{\mathcal{I}}f^\ast\mathcal{X},
\end{displaymath}
which is a weak equivalence if $f^\ast\mathcal{X}$ is a homotopy
colimit diagram. 

If $f$ is surjective, it reflects weak
equivalences, cf. Proposition \ref{prop:reflect}. This proves the claim.
\end{proof}

This implies that homotopy colimits in a model category of simplicial
sheaves can be checked on points, provided there are enough points,
cf. \cite[Proposition 3.1.10]{thesis}. 

\begin{corollary}
Let $\mathcal{E}$ be a topos with enough points, let $\mathcal{I}$ be
a small category, and let $\mathcal{X}:\mathcal{I}\rightarrow
\simplicial{\mathcal{E}}$ be a diagram. Then $\mathcal{X}$ is a
homotopy colimit diagram if and only if for each point $p$ of
$\mathcal{E}$ in a conservative set of points, the corresponding
diagram $p^\ast(\mathcal{X}):\mathcal{I}\rightarrow \simplicial{\set}$
is a homotopy colimit diagram. 
\end{corollary}

\begin{proof}
This follows from Proposition \ref{prop:colim2}: if $\mathcal{E}$ has
enough points, we can choose a conservative set $C$ of points, and
then the geometric morphism
\begin{displaymath}
\prod_{p\in C}\set\rightarrow \mathcal{E}
\end{displaymath}
is surjective.
\end{proof}

\subsection{Homotopy Pullbacks}

Finally, we recall the behaviour of homotopy pullbacks under inverse
images of geometric morphisms. As for homotopy colimits, they are
preserved by inverse images and reflected, provided the geometric
morphism is surjective. The argument does however not work for
arbitrary homotopy limits, since the inverse image fails to be a
right Quillen functor.

\begin{proposition}
\label{prop:pullbackloc}
Let $\mathcal{E}$ be a topos, let $f:\mathcal{E}'\rightarrow
\mathcal{E}$ be a geometric morphism, and let the following
commutative diagram $\mathcal{X}$ in $\simplicial{\mathcal{E}}$ be
given: 
\begin{center}
  \begin{minipage}[c]{10cm}
    \xymatrix{
      A \ar[r] \ar[d] & B \ar[d] \\
      C \ar[r] & D.
    }
  \end{minipage}
\end{center}
If $\mathcal{X}$ is a homotopy pullback diagram in
$\simplicial{\mathcal{E}}$, then $f^\ast\mathcal{X}$ is a homotopy
pullback diagram in $\simplicial{\mathcal{E}'}$.  
If moreover $f$ is  surjective, and $f^\ast\mathcal{X}$ is a
homotopy pullback diagram in $\simplicial{\mathcal{E}'}$, then
$\mathcal{X}$ is a homotopy pullback diagram in $\simplicial{\mathcal{E}}$.
\end{proposition}

\begin{proof}
The first assertion, i.e. that homotopy pullback squares are preserved
by the inverse image part of a geometric morphism is proved in
\cite[Theorem 1.5]{rezk:1998:sharp}. 

  Recall 
  that $\diagram{X}$ is a homotopy pullback diagram if there exists a 
  factorization of $f:B\rightarrow D$ into a trivial cofibration
  $i:B\rightarrow \tilde{B}$ and a fibration $g:\tilde{B}\rightarrow
  D$, such that the induced morphism $A\rightarrow C\times_D
  \tilde{B}$ is a weak equivalence. Since $f$ is surjective, it
  suffices to show that the induced morphism
  $f^\ast(A)\rightarrow f^\ast(C\times_D \tilde{B})\cong
  f^\ast(C)\times_{f^\ast(D)} f^\ast(\tilde{B})$ is a weak
  equivalence. Note that geometric morphisms preserve finite limits by
  definition, 
  which explains the last isomorphism. 

Consider the diagram
\begin{center}
  \begin{minipage}[c]{10cm}
    \xymatrix{
      f^\ast(A) \ar[r] \ar[d] & f^\ast(B)\ar[d] \\
      f^\ast(C)\times_{f^\ast(D)}f^\ast(\widetilde{B}) \ar[r] \ar[d] &
      f^\ast(\widetilde{B}) \ar[d] \\
      f^\ast(C) \ar[r] & f^\ast(D).
    }
  \end{minipage}
\end{center}
Since homotopy pullbacks are preserved by geometric morphisms, the
lower square is a homotopy pullback. By assumption, the outer square
is also homotopy pullback square, therefore the upper square is a
homotopy pullback, cf. Lemma \ref{lem:hopullback}. Since $f$ preserves
weak equivalences, $f^\ast(B)\rightarrow f^\ast(\widetilde{B})$ is a
weak equivalence. Therefore, the morphism $f^\ast(A)\rightarrow
f^\ast(C)\times_{f^\ast(D)}f^\ast(\widetilde{B})$ is also a weak
equivalence. This proves the result.

\end{proof}

As for homotopy colimits, we find that homotopy pullbacks in a
category of simplicial sheaves can be checked on points, provided
there are enough points, cf. \cite[Proposition 3.1.11]{thesis}.

\begin{corollary}
Let $\mathcal{E}$ be a topos with enough points, and let the following
commutative diagram $\diagram{X}$ of simplicial sheaves in
$\simplicial{\mathcal{E}}$ be given: 
\begin{center}
  \begin{minipage}[c]{10cm}
    \xymatrix{
      A \ar[r] \ar[d] & B \ar[d]^f \\
      C \ar[r] & D.
    }
  \end{minipage}
\end{center}
This is a homotopy pullback diagram iff for each point $p$ of $T$ in a
conservative set of points, the diagram $p^\ast(\diagram{X})$ of
simplicial sets is a homotopy pullback diagram. 
\end{corollary}

\subsection{Homotopy Distributivity}

The results on homotopy limits and colimits from the previous section
can be used to give a simple proof of the following result of Rezk on
homotopy distributivity in categories of simplicial sheaves,
cf. \cite[Theorem 1.4]{rezk:1998:sharp}.
These results generalize various results on commuting homotopy
pullbacks and homotopy colimits known to hold for simplicial sets,
such as Mather's cube theorem and Puppe's
theorem, cf. \prettyref{cor:mather} and
\prettyref{prop:diagbase}. Moreover, homotopy  
distributivity allows the construction of classifying spaces for fibre 
sequences, cf. \cite{thesis}.

We begin by explaining the precise definition of homotopy distributivity,
which is a homotopical generalization of the usual infinite 
distributivity law which holds for topoi. It is a statement about
commutation of arbitrary small homotopy colimits with finite homotopy
limits. Since any finite homotopy limit can be constructed via
homotopy pullbacks, it suffices to check that homotopy pullbacks
distribute over arbitrary homotopy colimits. Most of the work on
homotopy distributivity is due to Rezk
\cite{rezk:1998:sharp}. 

The situation is the following. Let $\category{C}$ be a simplicial
model category, let $\category{I}$ be a small category, and let
$f:\diagram{X}\rightarrow \diagram{Y}$ be a morphism of
$\category{I}$-diagrams in $\category{C}$. The diagrams we are most
interested in are the following: 

For any $i\in\category{I}$, we have a commutative square
\begin{center}
  \begin{minipage}[c]{10cm}
    \begin{equation}
      \label{eq:distrib1}
      \xymatrix{
        \diagram{X}(i) \ar[r] \ar[d]_{f(i)} & \colim_\category{I}
        \diagram{X} \ar[d] \\ 
        \diagram{Y}(i) \ar[r] & \colim_\category{I} \diagram{Y}.
      }
    \end{equation}
  \end{minipage}
\end{center}

Moreover, for any $\alpha:i\rightarrow j$ in $\category{I}$ we have a 
commutative square
\begin{center}
  \begin{minipage}[c]{10cm}
    \begin{equation}
      \label{eq:distrib2}
      \xymatrix{
        \diagram{X}(i) \ar[r]^{\diagram{X}(\alpha)} \ar[d]_{f(i)} &
        \diagram{X}(j)
        \ar[d]^{f(j)} \\ 
        \diagram{Y}(i) \ar[r]_{\diagram{Y}(\alpha)} &
        \diagram{Y}(j). 
      }
    \end{equation}
  \end{minipage}
\end{center}

Now we are ready to state the definition of homotopy distributivity,
following \cite{rezk:1998:sharp}. 

\begin{definition}[Homotopy Distributivity]
In the above situation, we  say that $\category{C}$ 
satisfies \emph{homotopy distributivity} if for any morphism
$f:\diagram{X}\rightarrow \diagram{Y}$ of 
$\category{I}$-diagrams in  $\category{C}$ for which $\diagram{Y}$ is
a homotopy 
colimit diagram, i.e.  $\hocolim_\category{I} \diagram{Y}\rightarrow
\colim_\category{I} \diagram{Y}$ is a weak equivalence, the following
two properties hold: 
\begin{enumerate}[(HD i)]
\item If each square of the form (\ref{eq:distrib1}) is a homotopy
  pullback, then $\diagram{X}$ is a homotopy colimit diagram.
\item If $\diagram{X}$ is a homotopy colimit diagram, and each
  diagram of the form (\ref{eq:distrib2}) is a homotopy  pullback, then
  each diagram of the form (\ref{eq:distrib1}) is also a homotopy pullback.  
\end{enumerate}
\end{definition}

\begin{example}
The category $\simplicial{\set}$ of simplicial sets satisfies homotopy
distributivity. This follows e.g. from the work of Puppe
\cite{puppe:1974:hofib} and Mather \cite{mather:1976:hopullback}. 
\end{example}

More generally, homotopy distributivity holds for all model categories
of simplicial sheaves on a Grothendieck site and can be proven quite easily
if the site has enough points. We give a short and simple proof
of homotopy distributivity, based on the reflection of homotopy
colimits and pullbacks proved earlier.
The general statement and proof using Boolean localizations can be
found in \cite{rezk:1998:sharp}. 

\begin{proposition}
\label{prop:hodistrib}
Let $\mathcal{E}$ be a Grothendieck topos with enough points. Then
homotopy distributivity holds for the injective model structure on
$\simplicial{\mathcal{E}}$.  
\end{proposition}

\begin{proof}
Since there are enough points, there exists a  surjective
geometric morphism 
\begin{displaymath}
f:\prod_{p\in C}\set\rightarrow\mathcal{E},
\end{displaymath}
where $C$ is a conservative set of points. 
By Propositions \ref{prop:colim2} and \ref{prop:pullbackloc} the
properties of homotopy colimit resp. homotopy pullback diagrams can be
checked locally. The assertion then follows from homotopy
distributivity for simplicial sets.
\end{proof}

We next discuss two important consequences of homotopy distributivity
for model categories of simplicial sheaves. One is a generalization
of Mather's cube theorem \cite{mather:1976:hopullback}. The other 
generalizes a theorem of Puppe \cite{puppe:1974:hofib} on commuting
homotopy fibres and homotopy pushouts to simplicial sheaves.

\begin{corollary}[Mather's Cube Theorem]
\label{cor:mather}

Let $\mathcal{E}$ be any Grothendieck topos. Consider the following
diagram of simplicial objects in $\mathcal{E}$:
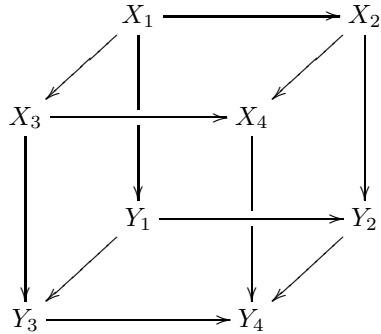
\begin{center}
  \begin{minipage}[c]{10cm}
    \xymatrix{
      & X_1 \ar'[d][dd]  \ar[dl] \ar[rr] & & X_2 \ar[dd] \ar[dl] \\
      X_3 \ar[dd] \ar[rr] & & X_4 \ar'[d][dd] \\
      & Y_1 \ar[dl] \ar[rr] & & Y_2 \ar[dl] \\
      Y_3 \ar[rr] & & Y_4
    }
  \end{minipage}
\end{center}
Assume that the bottom face, i.e. the one consisting of the
spaces $Y_i$, is a homotopy pushout, and that all the vertical 
faces are homotopy pullbacks. Then the top face is a homotopy
pushout. 

Moreover, taking the homotopy fibre commutes with homotopy
pushouts: for a commutative diagram
\begin{center}
  \begin{minipage}[c]{10cm}
    \xymatrix{
      E_2 \ar[d]_{p_2} & E_0 \ar[l] \ar[d]^{p_0} \ar[r] & E_1
      \ar[d]^{p_1} \\  
      B_2 & B_0 \ar[l] \ar[r]& B_1 
    }
  \end{minipage}
\end{center}
in which the squares are homotopy pullbacks, we have weak equivalences 
\begin{displaymath}
\hofib p_i \xrightarrow{\phantom{\cong}\cong\phantom{\cong}} \hofib
(p:E_1\cup^h_{E_0} 
E_2 \rightarrow B_1 \cup^h_{B_0} B_2).
\end{displaymath}
\end{corollary}

\begin{proof}
This is a consequence of homotopy distributivity,
cf. Proposition \ref{prop:hodistrib},
applied to homotopy pushout diagrams. The assumption in the definition
of homotopy distributivity is that the bottom face is a homotopy
colimit diagram, i.e. a homotopy pushout. 

For the first assertion, we note that since all the vertical faces are 
homotopy pullbacks, the diagonal square in the cube consisting of
$X_1$, $Y_1$, $X_4$ and $Y_4$ is also a homotopy pullback, by the
homotopy pullback lemma \ref{lem:hopullback}.
By (HD i) we conclude that the top square is a homotopy colimit
diagram, i.e. $X_4$ is weakly equivalent to the homotopy pushout
$X_2\cup^h_{X_1}X_3$.  
The restriction in the definition of
homotopy distributivity that $X_4$ be the point-set pushout of $X_2$
and $X_3$ along $X_1$ is not essential. Without loss of generality we
can assume that the morphisms $X_1\rightarrow X_3$
resp. $Y_1\rightarrow Y_3$ are cofibrations, and that $X_4$
resp. $Y_4$ are point-set pushouts. If this is not the case, just
replace the morphism by cofibrations, and obtain a cube which is
weakly equivalent to the cube we started with.

For the second statement note that since the squares are homotopy
pullbacks, we have $\hofib p_0\cong\hofib p_1\cong\hofib p_2$. Factoring
$E_0\rightarrow E_1$ resp. $B_0\rightarrow B_1$ as a cofibration
followed by a trivial fibration, we can assume that these morphisms
are cofibrations. 
Denote
$E=E_1\cup^h_{E_0} E_2$ and $B=B_1\cup^h_{B_0}B_2$. Then we are in the
situation to apply (HD ii). This implies that all the squares 
\begin{center}
  \begin{minipage}[c]{10cm}
    \xymatrix{
      E_i \ar[r] \ar[d] & E \ar[d] \\
      B_i \ar[r] &B
    }
  \end{minipage}
\end{center}
are homotopy pullback squares. In particular, we get the desired
weak equivalences $\hofib p_i\cong\hofib p$. 
\end{proof}

The following is a version of Puppe's Theorem \cite{puppe:1974:hofib}
for simplicial sheaves: 

\begin{proposition}[Puppe's Theorem]
\label{prop:diagbase}
Let $\mathcal{E}$ be a Grothendieck topos, and let
$\diagram{X}:\category{I}\rightarrow \simplicial{\mathcal{E}}$ be a 
diagram of simplicial objects over a fixed base simplicial object 
$Y$, i.e. the following diagram commutes for every
$\alpha:i\rightarrow j$ in $\category{I}$: 
\begin{center}
  \begin{minipage}[c]{10cm}
    \xymatrix{
      \diagram{X}(i) \ar[rr]^{\diagram{X}(\alpha)} \ar[rd] & &
      \diagram{X}(j) \ar[ld] \\ 
      & Y
    }
  \end{minipage}
\end{center}
There is an associated diagram of homotopy fibres 
\begin{displaymath}
\diagram{F}:\category{I}\rightarrow\simplicial{\mathcal{E}}:i\mapsto 
\hofib (\diagram{X}(i)\rightarrow Y)
\end{displaymath} 
Denoting $X=\hocolim_{\category{I}}\diagram{X}$ and
$F=\hocolim_{\category{I}}\diagram{F}$, we have a weak
equivalence $\hofib (X\rightarrow Y)\simeq F$.
\end{proposition}

\begin{proof}
We construct a new morphism of diagrams $\diagram{G}\rightarrow
\diagram{X}$, where the diagram $\diagram{G}$ is defined by
\begin{displaymath}
\diagram{G}:\category{I}\rightarrow \simplicial{\mathcal{E}}:i\mapsto 
\diagram{X}(i)\times^h_X \hofib(X\rightarrow Y).
\end{displaymath}
Without loss of generality we can assume $\hofib(X\rightarrow
Y)\rightarrow X$ is a fibration. Then the homotopy pullbacks above are
ordinary pullbacks, and $\colim\diagram{G}\cong \hofib(X\rightarrow Y)$. 
We apply the homotopy pullback lemma to the following diagram:
\begin{center}
  \begin{minipage}[c]{10cm}
    \xymatrix{
      \diagram{X}(i)\times^h_X \hofib(X\rightarrow Y) \ar[r] \ar[d] & 
      \diagram{X}(i) \ar[d] \\  
      \hofib(X\rightarrow Y) \ar[d] \ar[r] & X \ar[d] \\
      \pt \ar[r] & Y
    }
  \end{minipage}
\end{center}
This implies the following weak equivalence
\begin{displaymath}
\diagram{X}(i)\times^h_X \hofib(X\rightarrow Y)\simeq
\hofib(\diagram{X}(i)\rightarrow Y)=\diagram{F}(i).
\end{displaymath}
Invariance of homotopy colimits under weak equivalence,
cf. Proposition \ref{prop:hoinvariance}, implies a weak equivalence
$\hocolim\diagram{G}\cong\hocolim\diagram{F}$. 
Homotopy distributivity applied to the projection morphism
$\diagram{G}\rightarrow \diagram{X}$ implies that $\diagram{G}$ is a
homotopy colimit diagram. Putting everything together we obtain weak
equivalences $\hofib(X\rightarrow
Y)\cong\colim\diagram{G}\simeq\hocolim \diagram{G}\simeq \hocolim
\diagram{F}$, whence the desired statement follows.
\end{proof}

\subsection{Ganea's Theorem}

It is now possible to obtain some fibre sequences for simplicial
sheaves, which are known to hold for simplicial sets by homotopy
distributivity.
In the case of simplicial sets, these fibre sequences are all more or
less consequences of Ganea's work \cite{ganea:1965:suspension}. Their
simplicial sheaf analogues have been used in \cite{toric} to provide
partial  descriptions of the $\Ao$-fundamental group of smooth toric
varieties.   

 We start out with a theorem describing the homotopy fibre of the fold
 map. The proof is essentially the one given in \cite[Appendix
  HL]{farjoun:1996:cellular}, which simply applies homotopy
distributivity to one of the simplest situations possible: 

\begin{proposition}[Ganea's Theorem]
\label{prop:ganea1}
Let $\mathcal{E}$ be a Grothendieck topos, and let $X$ be a simplicial
object in $\mathcal{E}$. 
The sequence $\Sigma\Omega X\rightarrow X\vee X \rightarrow X$ is a
fibre sequence in $\simplicial{\mathcal{E}}$. 
\end{proposition}

\begin{proof}
This is an instance of Proposition \ref{prop:diagbase} applied to the
diagram: 
\begin{center}
  \begin{minipage}[c]{10cm}
    \xymatrix{
      X \ar[d]_= & \pt \ar[l] \ar[d] \ar[r] & X\ar[d]^= \\
      X & X\ar[l]^= \ar[r]_= &X
    }
  \end{minipage}
\end{center}
We are taking the homotopy colimit of the diagram over the fixed base
space $X$, and the homotopy colimit of the upper line yields $X\vee
X$. The map to $X$ is the fold map $\vee:X\vee X\rightarrow X$. Then
Proposition \ref{prop:diagbase} shows that the fibre is given by the
homotopy colimit of the diagram of fibres:
\begin{displaymath}
\pt \longleftarrow \Omega X \longrightarrow \pt.
\end{displaymath}
This is by definition $\Sigma\Omega X$.
\end{proof}

\begin{example}
\label{ex:projline}
A particular topological instance of the above is the fibre sequence 
\begin{displaymath}
S^2\rightarrow \mathbb{C}P^\infty\vee \mathbb{C}P^\infty \rightarrow
\mathbb{C}P^\infty. 
\end{displaymath}
A similar fibre sequence exists in $\simplicial{\topos{\sms}}$ with
any of the usual topologies.
This implies that there is a fibre sequence
\begin{displaymath}
\Sigma^1_s\mathbb{G}_m\rightarrow B\mathbb{G}_m\vee B\mathbb{G}_m
\rightarrow B\mathbb{G}_m.
\end{displaymath}
$\Ao$-locally, this yields a fibre sequence
\begin{displaymath}
\mathbb{P}^1\rightarrow \mathbb{P}^\infty\vee
\mathbb{P}^\infty\rightarrow \mathbb{P}^\infty.
\end{displaymath}
\end{example}

There are also other fibre sequences one can obtain: 
By considering similar diagrams as in \cite[Appendix
  HL]{farjoun:1996:cellular} we get the following fibration
sequences in any model category of simplicial sheaves. In the next
proposition, $X \ast Y$ denotes the join of $X$ and $Y$ which is
defined as the homotopy pushout of the diagram $X\leftarrow X\times
Y\rightarrow Y$.

\begin{proposition}
\label{prop:ganea2}
Let $\mathcal{E}$ be a Grothendieck topos, and $X$ be a simplicial
object in $\mathcal{E}$. 
The sequence $\Omega X_0\ast\Omega X_1\rightarrow X_0\vee X_1
\rightarrow X_0\times X_1$ is a 
fibre sequence in $\simplicial{\mathcal{E}}$. 
\end{proposition}

\begin{proof}
Apply Puppe's theorem \ref{prop:diagbase} to the following diagram,
the horizontal lines are the pushout diagrams and the vertical lines
are fibre sequences:
\begin{center}
  \begin{minipage}[c]{10cm}
    \xymatrix{
      \Omega X_0\times \ast \ar[d] & \Omega X_0\times \Omega X_1
      \ar[d] \ar[r] \ar[l] & \ast\times \Omega X_1 \ar[d] \\
      \ast\times X_1 \ar[d] & \ast \ar[l] \ar[d] \ar[r] & X_0\times\ast
      \ar[d] \\
      X_0\times X_1 & X_0\times X_1 \ar[l] \ar[r] & X_0\times X_1.
    }
  \end{minipage}
\end{center}
\end{proof}

\begin{example}
\label{ex:projline2}
An instantiation of the above fibre sequence similar to the one given in
Example \ref{ex:projline} is the following fibre sequence in
$\simplicial{\topos{\sms}}$: 
\begin{displaymath}
\mathbb{G}_m\ast\mathbb{G}_m
\rightarrow B\mathbb{G}_m\vee B\mathbb{G}_m\rightarrow
B\mathbb{G}_m \times B\mathbb{G}_m.
\end{displaymath}
$\Ao$-locally, this yields a fibre sequence
\begin{displaymath}
\mathbb{A}^2\setminus\{0\}\rightarrow \mathbb{P}^\infty\vee
\mathbb{P}^\infty \rightarrow \mathbb{P}^\infty\times \mathbb{P}^\infty.
\end{displaymath}
\end{example}

As a final example, we restate yet another theorem of Ganea 
\cite{ganea:1965:suspension}. It should  by now be obvious, which
diagram to apply Puppe's theorem to. 

\begin{proposition}
\label{prop:ganea3}
Let $\mathcal{E}$ be a Grothendieck topos, and let $F\rightarrow
E\rightarrow B$  be any fibre sequence of simplicial objects. Then
there is another fibre sequence 
\begin{displaymath}
F\ast\Omega B \longrightarrow E\cup CF=E/F \longrightarrow B.
\end{displaymath}
\end{proposition}

\subsection{Canonical Homotopy Colimit Decomposition} 

Let $p:E\rightarrow B$ be a fibration of fibrant simplicial sets. Then
the canonical homotopy colimit decomposition of $B$ allows to write
$B$ as homotopy colimit of standard simplices $\Delta^n\rightarrow
B$. Then we can pull back the fibration $p$ to these simplices and
obtain the homotopy fibres. By homotopy distributivity, $E$ can be
written as the homotopy colimit over the simplex category
$\Delta\downarrow B$ of the homotopy fibres. 
The same statement works for simplicial sheaves:
The right notion to formulate it is the canonical
homotopy colimit decomposition for objects in a combinatorial model
category, which was described in detail in
\cite{dugger:2001:combinatorial}. 

Let $\category{M}$ be a combinatorial model category, $\category{C}$
be a small category. For any functor
$I:\category{C}\rightarrow\category{M}$ and a fixed cosimplicial
resolution $\Gamma_I:\category{C}\rightarrow\Delta\category{M}$, we
obtain a functor
$\category{C}\times\Delta\rightarrow\category{M}:(U,[n])\mapsto 
\Gamma(n)(U)$.  For any object $X$, we can consider the over-category
(resp. comma category in Mac Lane's terminology \cite[Section
  II.6]{maclane:1998:categories}) 
$(\category{C}\times\Delta\downarrow X)$ and the canonical diagram
$(\category{C}\times\Delta\downarrow X)\rightarrow
\category{M}:\Gamma(n)(U)\mapsto U\times\Delta^n$. 

\begin{lemma}
\label{lem:decompfib}
Let $T$ be a site, and let 
$p:E\rightarrow B$ be a fibration of fibrant simplicial sheaves. Then
$p$ is weakly   equivalent to the morphism of simplicial sheaves
\begin{displaymath}
\hocolim\diagram{F}\rightarrow \hocolim(T\times \Delta\downarrow B),
\end{displaymath}
where $(T\times\Delta\downarrow B)$ is the canonical diagram
associated to some fixed cosimplicial resolution, and the diagram
$\diagram{F}$ is the diagram of homotopy fibres: the index category is
still $(T\times\Delta\downarrow B)$, but an object
$U\times\Delta^n\rightarrow B$ is mapped to the pullback
$(U\times\Delta^n)\times_B E$, which is the fibre of $p$ over $U$.  
\end{lemma}

This is not as useful as the same construction for simplicial sets,
since the homotopy types of the various $U\in T$ are
different, which is the same as saying that a simplicial sheaf is
not locally contractible. Therefore, not all of the simplicial
sheaves $(U\times\Delta^n)\times_B E$ are weakly equivalent.

\section{Preliminaries on Fibre Sequences}
\label{sec:fibseq} 

We first repeat the definition of fibre sequences in model
categories, taken from \cite{hovey:1998:modelcats}. For details
of the proof see \cite[Theorem 6.2.1]{hovey:1998:modelcats}.

\begin{definition}
\label{def:fibaction}
Given a fibration $p:E\rightarrow B$ of fibrant objects with fibre
$i:F\rightarrow E$. There is an action of $\Omega B$ on $F$, given as
follows. Let 
$h:A\times I\rightarrow B$ represent $[h]\in[A,\Omega B]$ and let
$u:A\rightarrow F$ represent $[u]\in[A,F]$. We define $\alpha:A\times
I\rightarrow E$ as the lift in the following diagram:
\begin{center}
  \begin{minipage}[c]{10cm}
    \xymatrix{
      A \ar[r]^{i\circ u} \ar[d]_{i_0} & E \ar[d]^p \\
      A\times I \ar[r]_h & B
    }
  \end{minipage}
\end{center}
Then define $[u].[h]=[w]$ with $w:A\rightarrow F$ to be the unique map 
satisfying $i\circ w=\alpha\circ i_1$. 

This defines a natural right action of $[A,\Omega B]$ on $[A,F]$ for
any $A$, which  suffices to provide an action of $\Omega B$ on $F$. 
\end{definition}

Note that the action of $\Omega B$ on $F$ is an action in the homotopy
category: $\Omega B$ acts on $F$ only up to homotopy since the action
is defined by using the homotopy lifting property,
cf. \cite{stasheff:1974}. 

This motivates the definition of fibre sequences \cite[Definition
  6.2.6]{hovey:1998:modelcats}, given as follows:

\begin{definition}
\label{def:fibreseq}
Let $\category{C}$ be a pointed model category. 
A \emph{fibre sequence} is a diagram $X\rightarrow Y\rightarrow Z$
together with a right action of $\Omega Z$ on $X$ that is isomorphic
in $\operatorname{Ho}\category{C}$ to a diagram
$F\stackrel{i}{\longrightarrow} E\stackrel{p}{\longrightarrow} B$
where $p$ is a fibration of fibrant objects with fibre $i$ and $F$ has
the right $\Omega B$-action of \prettyref{def:fibaction}. 
\end{definition}

The following proposition shows that fibrations induce fibre sequences
in the sense of of \prettyref{def:fibreseq}.

\begin{proposition}
\label{prop:properfibre}
Let $\category{C}$ be a proper pointed model category. Let
$p:E\rightarrow B$ be a fibration, and denote by 
$F$ a cofibrant replacement of $p^{-1}(\ast)$. Then 
$F\rightarrow E\stackrel{p}{\longrightarrow} B$ is a   fibre
sequence. 
\end{proposition}

\subsection{Locally trivial morphisms}

Already in the case of simplicial sets, one has to restrict the
classification problem for fibrations to obtain a classifying
space. One possible such restriction is to consider only base spaces
$B$ which are connected. Another approach is to consider only
fibrations $p:E\rightarrow B$ for which the fibres $p^{-1}(b)$ have
the weak homotopy type of $F$ for all $b\in B$. 

Also in the simplicial sheaf case, we need such a restriction. 
The obvious way to define connectedness for simplicial sheaves is
the one used e.g. in 
\cite[Corollary 2.3.22]{morel:voevodsky:1999:a1}. 
\begin{definition}
\label{def:conn}
Let $X$ be a pointed simplicial sheaf on a Grothendieck site
$T$. We say that $X$ is \emph{connected} if $L^2\pi_0 X=\ast$, where
$L^2$ denotes sheafification. In other words, for any point $x$ of the
topos  $\topos{T}$, we require that the simplicial set $x^\ast(X)$ is
connected.
\end{definition}

The main difference to the topological notion of connectedness is that 
a topological space is always the disjoint union of its connected
components. This is no longer true for simplicial sheaves. The
representable sheaves of a site can be viewed as constant
simplicial sheaves; usually they are neither connected in the above
sense nor decomposable into a direct sum of connected sheaves. 

The topological way out of the connectivity problem therefore  becomes
a little awkward. We will consider a different type of 
condition which makes sure that the fibre sequences over a general
simplicial sheaf form a set (at least after passing to equivalence
classes).  This is done by introducing local triviality with respect
to a Grothendieck topology -- the least common denominator of the
algebraic topology and algebraic geometry usage of terms like
fibration. 

\begin{definition}
Let $T$ be a Grothendieck site. 
We say that a morphism $p:E\rightarrow B$ of simplicial sheaves is
\emph{locally trivial with fibre $F$}, if for each object $U$ in $T$
and each morphism $U\times\Delta^n\rightarrow B$, there exists a covering
$\bigsqcup U_i\rightarrow U$ such that there are weak equivalences
\begin{displaymath}
E\times_B (U_i\times\Delta^n)\simeq F\times (U_i\times\Delta^n).
\end{displaymath}
\end{definition}

\begin{example}
As an example, consider the category of smooth manifolds with the
Grothendieck topology generated by the open coverings. 
A fibre sequence $F\rightarrow E\rightarrow B$ is locally trivial if
for each pullback $E\times_BM\rightarrow M$ of this sequence to a
smooth manifold $M$, there exists a covering $\bigsqcup U_i\rightarrow
M$ of $M$ by open submanifolds such that $U_i\times_BE\simeq F$. But a
fibration $E\rightarrow M$ over a smooth (connected) manifold $M$ is
always locally trivial in this sense: for each contractible open 
submanifold $U$ of $M$, then $U\times_ME$ is weakly equivalent to the
point set fibre over any point of $U$. 
Therefore, fibrations of connected topological spaces are indeed
locally trivial in the above sense. 

Note also that the local triviality condition forces all points to
have fibres weakly equivalent to $X$. This shows that the above local
triviality condition reduces to the usual assumptions used e.g. in
\cite{allaud:1966:fibre}. 
\end{example}

We remark that the results discussed in \prettyref{sec:hocolim} are an
analogue of the theory of quasi-fibrations,
cf. \cite{dold:thom,dold:lashof}. In fact, we have the following:

\begin{proposition}
Let $p:E\rightarrow B$ be a locally trivial morphism of pointed
simplicial sheaves with fibre $F=p^{-1}(\ast)$. Then $F\rightarrow 
E\rightarrow B$ is a fibre sequence in the sense of
\prettyref{def:fibreseq}. 
\end{proposition}

\section{First Variant: Brown Representability}
\label{sec:one}

In this section, we will construct classifying spaces of fibre
sequences via the Brown representability theorem. For topological
spaces, this approach was used by Allaud \cite{allaud:1966:fibre},
Dold \cite{dold:1966}, and Sch\"on \cite{schoen:1982}. 
A textbook treatment of this approach can be found in
\cite{rudyak:1998:cobordism}.

\subsection{Fibre Sequences Functor}

\newcommand{\fibfun}[1]{\ensuremath{\mathcal{H}^{\operatorname{pt}}({#1},F)}}

We now define the functor mapping a simplicial sheaf to the \emph{set}
of fibre sequences with fixed fibre over this simplicial sheaf. We
will work in the injective 
model category of \emph{pointed} simplicial sheaves on some site 
$T$. This is due to the fact that fibre sequences as in 
\prettyref{def:fibreseq} are only defined in pointed model
categories. Moreover, the Brown representability theorem also requires
pointed model categories. There are examples in
\cite{heller:1981:brown} showing that Brown representability might
fail already for unpointed topological spaces. 

\begin{definition}
\label{def:fibfun}
Recall from \prettyref{def:fibreseq} that a fibre sequence over
$X$ with fibre $F$ is a diagram $F\rightarrow E
\stackrel{p}{\longrightarrow} X$ with an $\Omega X$-action on $F$ which is
isomorphic in the homotopy category to the fibre sequence associated
to a fibration $p:\tilde{E}\rightarrow \tilde{X}$ of fibrant
replacements $\tilde{X}$ of $X$ and $\tilde{E}$ of $E$. Up to
isomorphism in the homotopy category, we will usually assume that
our fibre sequence $F\rightarrow E \stackrel{p}{\longrightarrow} X$ is
represented by some actual fibration over some fibrant replacement
$\tilde{X}$ of $X$. 

A \emph{morphism of fibre sequences} is a diagram in
$\simplicial{\topos{T}}$ 
\begin{center}
  \begin{minipage}[c]{10cm}
    \xymatrix{
      F_1 \ar[d]_f \ar[r] & E_1 \ar[d]_g \ar[r]^{p_1} & B_1 \ar[d]^h \\
      F_2 \ar[r] & E_2 \ar[r]_{p_2} & B_2,
    }
  \end{minipage}
\end{center}
such that the left square commutes up to homotopy, and the right
square is commutative, and $f$ is $\Omega h$-equivariant, i.e. the
following diagram is homotopy commutative:
\begin{center}
  \begin{minipage}[c]{10cm}
    \xymatrix{
      \Omega B_1\times F_1 \ar[r] \ar[d]_{\Omega h\times f} & F_1
      \ar[d]^f \\
      \Omega B_2\times F_2\ar[r] & F_2.
    }
  \end{minipage}
\end{center}

This in particular allows to define what an equivalence of fibre
sequences over $X$ is:
Two fibre sequences over $X$ with fibre $F$ are equivalent if there is
an isomorphism of fibre sequences 
\begin{center}
  \begin{minipage}[c]{10cm}
    \xymatrix{
      F \ar[d]_\id \ar[r] & E_1 \ar[d] \ar[r] & X
      \ar[d]^{\id}\\ 
      F \ar[r] & E_2 \ar[r] & X,
    }
  \end{minipage}
\end{center}
in the homotopy category
$\operatorname{Ho}\simplicial{\topos{T}}$. We denote this by $E_1\sim 
E_2$.  
\end{definition}

\begin{remark}
\label{rem:rem1}
\begin{enumerate}[(i)]
\item The following can be assumed without loss of generality:
  we can assume that the base $B$ is fibrant, that the morphism $p$ is
  a fibration, and that $F$ is the point-set fibre of $p$ over
  $\ast\hookrightarrow B$. This basically follows from 
  \prettyref{prop:properfibre}.   
\item Note that in the definition of a morphism of fibre sequences we
  can always arrange for the right square to be commutative on the
  nose. We just lift the morphism $h\circ p_1$ along the fibration
  $p_2$. This makes the right square commutative, and leaves the left
  square commutative up to homotopy. 
\item In the case of topological spaces, the above definition was used
  by Allaud, cf. \cite{allaud:1966:fibre}. It coincides with the
  notion of fibre homotopy equivalence by a theorem of Dold,
  cf. \cite[Theorem 6.3]{dold:1963:fib}.
\end{enumerate}
\end{remark}

\begin{lemma}
Equivalence of fibre sequences is an equivalence relation.
\end{lemma}

\begin{proof}
This is clear since equivalence was defined by isomorphism in the
homotopy category, which implies reflexivity, symmetry and
transitivity. 
%
\end{proof}

\begin{definition}[Pullback of Fibre Sequences]
\label{def:pbfibfun}
Let $f:B_1\rightarrow B_2$ be a pointed map, and let
$F\rightarrow E_2\rightarrow B_2$ be a fibre sequence. We define a fibre
sequence with fibre $F$ over $B_1$ as follows:
\begin{center}
  \begin{minipage}[c]{10cm}
    \xymatrix{
      & F \ar[ld]_a \ar'[d][dd] \ar[rd] \\
      E_1 \ar[rr] \ar[dd]_{p_1}  && E_2 \ar[dd]^{p_2} \\
      & \ast \ar[ld] \ar[rd] \\
      B_1 \ar[rr]_f && B_2
    }
  \end{minipage}
\end{center}
We assume that $p_2$ is a fibration, and define
$E_1$ as the pullback $E_2\times_{B_2} B_1$ of $p_2$ along $f$. 
Note that $E_1$ is therefore also the homotopy pullback of $p_2$ along
$f$, and $p_1$ is a fibration. By the universal property of pullbacks,
we have a morphism $a:F\rightarrow E_1$. Moreover, by the pullback
lemma we have $p_1^{-1}(\ast)=F$, and since $p_1$ is a fibration, this
is also the homotopy fibre.  
\end{definition}

Let $T$ be a Grothendieck site. For given pointed simplicial
sheaves $X$ and $F$ on $T$, let $\fibfun{X}$ denote the collection
of equivalence classes of locally trivial fibre sequences over $X$
with fibre $F$ modulo the equivalence $\sim$. We want to show that
this is a 
set. 

\begin{proposition}
\label{prop:fibfunset}
For any $X,F\in\simplicial{\topos{T}}_\ast$, the collection
$\fibfun{X}$ is a set. Hence, with the pullbacks as in
\prettyref{def:pbfibfun}, we have a functor
\begin{displaymath}
\fibfun{-}:\simplicial{\topos{T}}_\ast\rightarrow \set_\ast.
\end{displaymath}
The natural base point of $\fibfun{X}$ is given by the trivial fibre
sequence $F\rightarrow X\times F\rightarrow X$, where the first map is
inclusion via the base point $\ast\rightarrow X$, and the second is
the product projection.
\end{proposition}

\begin{proof}
We first show that for every simplicial sheaf $X$ there is only a
set of equivalence classes of fibre sequences $F\rightarrow
E\rightarrow X$. We follow the lines of \cite[Theorem
  IV.1.55]{rudyak:1998:cobordism}. Note that this includes forward
references to \prettyref{prop:fibfunhoinv},
\prettyref{prop:fibfunwedge} and \prettyref{prop:fibfunmv}.

We start with the case of fibre sequences $F\rightarrow E\rightarrow
U$ for $U\in T$ viewed as constant simplicial sheaf. We assume
$E\rightarrow U$ is actually a fibration. For the
above fibre sequence, there exists a covering $U_i$ of $U$ such
that $F\rightarrow E\times_U U_i\rightarrow_U U_i$ is a trivial fibre
sequence with given trivializations $E\times_U U_i\simeq F\times U_i$
and transition morphisms $F\times (U_i\times_U U_j)\rightarrow F\times
(U_i\times U_j)$ which are weak equivalences. 
Now \prettyref{prop:hodistrib} implies that the original fibre
sequence $F\rightarrow E\rightarrow U$ can be reconstructed up to
equivalence as the homotopy colimit 
\begin{displaymath}
F\rightarrow \hocolim E_i\rightarrow \hocolim U_i.
\end{displaymath}
The Grothendieck site is (essentially) small, so there is only a set
of coverings, and for a given covering, there is only a set of
possible transition morphisms. The set of all locally trivial fibre
sequences up to equivalence is therefore contained in the product of
the sets of all possible transition morphisms (indexed by the possible
coverings of $U$). It is therefore a set. 

Next, we extend this result to simplicial sheaves of the form
$U\times\Delta^n$ for $U\in T$. The argument in
\prettyref{prop:fibfunhoinv} is independent of the $\fibfun{-}$ being
sets. It therefore shows that for a weak equivalence $f:X\rightarrow
Y$,  if $\fibfun{X}$ is a set, then so is $\fibfun{Y}$. We find that
for any simplicial sheaf of the form $U\times \Delta^n$ with $U\in
T$, $\fibfun{U\times\Delta^n}$ is a set. 

Finally, we use the decomposition of fibrations over the canonical
homotopy colimit presentation of the base simplicial sheaf,
cf. \prettyref{lem:decompfib}. We consider $F$-fibre sequences over
$B$, and decompose $B$ as a homotopy colimit over the category of
simplices $(T\times\Delta\downarrow B)$. The simplicial sheaves
indexed by this diagram are of the form $U\times\Delta^n$ for $U\in
T$, and we have already shown that fibre sequences over these form a
set. Moreover, the site $T$ is (essentially) small, therefore the
diagram is set-indexed. 

We now have to show that for any set-indexed homotopy colimit
$\hocolim_\alpha X_\alpha$ of spaces $X_\alpha$ for which $\fibfun{X_\alpha}$
is a set, the collection
\begin{displaymath}
\fibfun{\hocolim_\alpha X_\alpha}
\end{displaymath}
is also a set. Since all homotopy colimits can be decomposed into
homotopy pushouts and wedges, it suffices to show this assertion for
these special homotopy colimits. 

The proof of \prettyref{prop:fibfunwedge} shows
that if $\fibfun{X_\alpha}$ is a set for a set-indexed collection
$X_\alpha$, then $\fibfun{\bigvee_\alpha X_\alpha}$ is also a set.

For the homotopy pushouts, we use the proof of
\prettyref{prop:fibfunmv}. We get a surjective morphism of classes
\begin{displaymath}
d:\fibfun{B_1\cup_{B_0} B_2} \twoheadrightarrow
\fibfun{B_1}\times_{\fibfun{B_0}} \fibfun{B_2}.
\end{displaymath}
By assumption, $\fibfun{B_1}\times_{\fibfun{B_0}} \fibfun{B_2}$ is a
set. The morphism $d$ decomposes a fibre sequence $E$ over
$B_1\cup_{B_0} B_2$ into the pullbacks of the fibre sequence $E$ to
$B_1$ resp. $B_2$. These fibre sequences are remembered in the element
in $\fibfun{B_1}\times_{\fibfun{B_0}} \fibfun{B_2}$. What is
forgotten, and what constitutes the kernel of $d$ is the isomorphism
between the pullbacks of $E$ to $B_1$ resp. $B_2$. Since there is only
a set of automorphisms for any given fibre sequence, the kernel of $d$
is also a set. This implies that $\fibfun{B_1\cup_{B_0} B_2}$ is also
a set. 

Therefore, $\fibfun{B}$ is a set for any simplicial sheaf $B$. 

We still need to check that the pullback is really well-defined. This
is a simple diagram check, using the cogluing lemma and
therefore needing properness:
assume we have two fibre sequences $E_1$ and $E_2$ over $B$,
which are isomorphic in the homotopy category. We may assume that
$p_i:E_i\rightarrow B$ are fibrations. If not, we choose
factorizations. 
The independence of the choice of such fibrant replacements is proven in
\prettyref{prop:help1}. 
We consider the pullback $E_i\times_B A$ of the fibre sequence $E_i$
along the morphism $f:A\rightarrow B$. The isomorphism in the homotopy
category lifts to a zig-zag of weak equivalences, so it suffices to
show that a weak equivalence $g:E_1\rightarrow E_2$ pulls back to a
weak equivalence $f:E_1\times_B A\rightarrow E_2\times_B A$. This
follows from the cogluing lemma \cite[Corollar
II.8.13]{goerss:jardine:1999:simplicial}.  

Finally, for any fibre sequence $F\rightarrow E\rightarrow B$, 
which is locally trivial in the $T$-topology the pullback
$F\rightarrow E\times_B B'\rightarrow B'$ along any morphism $B'$ is
again locally trivial. This follows by a simple argument from the
pullback lemma: the pullback of $E'=E\times_B B'$ along any morphism
$U\rightarrow B'$ for $U\in T$ is also the pullback of $E$ along
$U\rightarrow B'\rightarrow B$.
\end{proof}

\begin{proposition}
\label{prop:help1}
Let $T$ be a Grothendieck site. All spaces and maps appearing below
are in the category $\simplicial{\topos{T}}$ of simplicial sheaves on
$T$. 
\begin{enumerate}[(i)]
\item
For any commutative diagram
\begin{center}
  \begin{minipage}[c]{10cm}
    \xymatrix{
      E_1\ar[rr]^\simeq \ar[rd]_{p_1} && E_2 \ar[ld]^{p_2} \\
      &B
    }
  \end{minipage}
\end{center}
with $p_1$ and $p_2$ fibrations, the induced weak equivalence on the
fibres is equivariant for the $\Omega B$-action in the homotopy
category.

\item
Let $p:E\rightarrow B$ be any morphism with homotopy fibre $F=\hofib
p$. Then the class $[\tilde{p}]\in \fibfun{B}$ of a fibrant
replacement $\tilde{p}:\tilde{E}\rightarrow B$ of $p$ is independent
of the choice of fibrant replacement.  

\item
For any homotopy pullback 
\begin{center}
  \begin{minipage}[c]{10cm}
    \xymatrix{
      E_1 \ar[r] \ar[d]_p & E_2 \ar[d]^q \\
      B_1 \ar[r]_f & B_2,
    }
  \end{minipage}
\end{center}
we have $f^\ast [q]=[p]$.
\end{enumerate}
\end{proposition}

\begin{proof}
(i) This follows since the action as in 
  \prettyref{def:fibaction} is given by liftings in a diagram:
  \begin{center}
    \begin{minipage}[c]{10cm}
      \xymatrix{
        A \ar[r]^{i\circ u} \ar[dd]_{i_0} & E_1 \ar[d]^\simeq \\
        & E_2 \ar[d]^{p_2} \\
        A\times I\ar[r]_h \ar@{.>}[uur]_\theta & B
      }
    \end{minipage}
  \end{center}
  The action of $\Omega B$ on $u$ is given by the lift $\theta$ which
  factors   through the fibre. Lifting to $E_1$ and composing with the
  weak   equivalence $E_1\rightarrow E_2$ yields a lift for
  $E_2$. Therefore,   the induced weak equivalence of the fibres is
  equivariant for the   $\Omega B$-action.

(ii) Note that the injective model structure on simplicial sheaves
is a proper model category, see \prettyref{thm:jardine}. 
By \prettyref{prop:properfibre}, $F\rightarrow E\rightarrow B$
is a fibre sequence for $p$ a fibration. Now for an arbitrary map
$p:E\rightarrow B$, we define $[p]$ as the fibre sequence associated
to the fibration in a factorization 
\begin{center}
\begin{minipage}[c]{10cm}
\xymatrix{
  E \ar[r]^\simeq_i \ar[rd]_p & \tilde{E} \ar[d]^{\tilde{p}} \\ &B.
}
\end{minipage}
\end{center}
We need to prove that this is independent of the factorization. We
consider two replacements of $p$ by a fibration: $p_1:E_1\rightarrow
B$ and $p_2:E_2\rightarrow B$. Note that by construction we have
trivial cofibrations $E_1 \stackrel{\cong}{\longrightarrow} E$
and $E_2\stackrel{\cong}{\longrightarrow} E$. We then consider
the lift in the following diagram:
\begin{center}
  \begin{minipage}[c]{10cm}
    \xymatrix{
      E \ar[r]^\simeq \ar[d]_\simeq & E_2 \ar[d]^{p_2} \\
      E_1 \ar[r]_{p_1} & B.
    }
  \end{minipage}
\end{center}
This is a weak equivalence by 2-out-of-3. As in the proof of
\prettyref{prop:properfibre} we obtain a weak equivalence
between the fibres $F_1$ and $F_2$. By (i), this morphism is
equivariant for the $\Omega B$-action, so $[p_1]=[p_2]$. 

(iii) Consider the homotopy pullback square in the statement of the
proposition. 
By \cite[Lemma II.8.16]{goerss:jardine:1999:simplicial}, there
exists a factorization of $q$ into a trivial cofibration
$i:E_2\rightarrow \tilde{E_2}$ and fibration
$\tilde{q}:\tilde{E_2}\rightarrow B_2$ such that the induced
morphism $E_1\rightarrow B_1\times_{B_2} \tilde{E_2}$ is a weak
equivalence. Note that $f^\ast[q]$ is given by
$B_1\times_{B_2}\tilde{E_2}$. By (ii) we can take any fibration
$\tilde{E_1}\rightarrow B_1$ with $E_1\rightarrow \tilde{E_1}$ a weak
equivalence, and by (i), the homotopy fibres of $\tilde{E_1}$ and
$B_1\times_{B_2}\tilde{E_2}$ are weakly equivalent and the weak
equivalence is $\Omega B$-equivariant. Therefore $f^\ast[q]=[p]$. 
\end{proof}

\subsection{Brown Representability}

The Brown representability theorem
is not really a single theorem, but rather a class of results stating
conditions under which a set-valued functor on a model or homotopy
category is representable. The first appearance is in the article of
Brown \cite{brown:1962:representability} in which it is proven that a
contravariant homotopy-continuous functor on the category of
topological spaces is representable, with main application to the
construction of spaces representing generalized cohomology theories.  
A more detailed analysis of why contravariant functors and pointed
model categories are necessary assumptions 
was done in \cite{heller:1981:brown}. 
Nowadays any reasonable textbook on algebraic topology contains a 
section on Brown representability for topological spaces. 

There are not so many results on Brown representability for general,
in particular unstable model categories. For a general model category,
Brown  representability usually fails, and at least some smallness
assumptions are necessary. In this paper, we use the 
representability theorem by Jardine for compactly generated model
categories, which is proven in \cite{jardine:2006:brown}.

There are several names for the condition on the functors. Functors
that satisfy the conditions for the representability were called
half-exact in \cite{brown:1962:representability}, but we use the term
homotopy-continuous. 
The terminology homotopy-continuous is reminiscent of Mac Lane's usage
of the term \emph{continuous} for a functor which preserves limits 
\cite[Section V.4]{maclane:1998:categories}. Homotopy-continuous
functors are the model category analogues of such continuous functors,
as the Brown representability is a version of the adjoint functor
theorem for model categories. 

\begin{definition}[Homotopy-Continuous Functor]
A functor
$F:\category{C}^{op}\rightarrow \set_\ast$ on a pointed model category
$\category{C}$ is called
\emph{homotopy-continuous} if it satisfies the following assumptions:
\begin{enumerate}[(HC i)]
\item $F$ takes weak equivalences to bijections.
\item $F(\pt)=\{\ast\}$.
\item For any coproduct $\bigvee_\alpha X_\alpha$ of a set
  $\{X_\alpha\}$ of objects of $\category{C}$ the following
  \emph{wedge axiom} is satisfied: 
  \begin{displaymath}
    F\big(\bigvee_\alpha X_\alpha\big)=\prod_\alpha F(X_\alpha).
  \end{displaymath}
\item For any homotopy pushout
  \begin{center}
    \begin{minipage}[c]{10cm}
      \xymatrix{
        A \ar[r] \ar[d] & X \ar[d]  \\
        B \ar[r] & Y,
      }
    \end{minipage}
  \end{center}
  the induced morphism is surjective:
  \begin{displaymath}
    F(Y)\twoheadrightarrow F(B)\times_{F(A)} F(X). 
  \end{displaymath}
  This is called the \emph{Mayer-Vietoris axiom}.
\end{enumerate}
\end{definition}




Now we recall Jardine's version of the Brown representability theorem
\cite[Theorem 19]{jardine:2006:brown}.
In this version, we need the following definition, cf. \cite[Section
3]{jardine:2006:brown}:

\begin{definition}
\label{def:cptjardine}
A model category $\mathcal{C}$ is called \emph{compactly generated},
if there is a set of compact cofibrant objects $\{K_i\}$ such that a
map $f:X\rightarrow Y$ is a weak equivalence if and only if it induces a
bijection 
\begin{displaymath}
[K_i,X]\xrightarrow{\phantom{X}\cong\phantom{X}} [K_i,Y]
\end{displaymath}
for all objects $K_i$ in the generating set. 
\end{definition}

This is a size condition that does not hold for all model categories
of simplicial sheaves. It is explained in \cite[Section 3,
p.88]{jardine:2006:brown} that the injective model structure on the
category simplicial sheaves on the Zariski resp. Nisnevich site is
compactly generated. 

\begin{theorem}[Brown Representability (after Jardine)]
\label{thm:brown2}
For a pointed, left proper, compactly generated model category
$\category{C}$ and a homotopy-continuous functor
$F:\category{C}^{op}\rightarrow \set_\ast$, there exists an object
$Y$ of $\category{C}$, a universal element $u\in F(Y)$, and a natural
isomorphism  
\begin{displaymath}
T_u:\Hom_{\operatorname{Ho}(\category{C})}(X, Y)
\xrightarrow{\phantom{X}\cong\phantom{X}} F(X): f\mapsto f^\ast(u)
\end{displaymath}
for any object $X$ of $\category{C}$. 
\end{theorem}

\begin{corollary}
\label{cor:brownmor}
Let $\category{C}$ be any left proper, compactly generated model
category, let $F, G:\category{C}^{op}\rightarrow \set_\ast$ be
homotopy-continuous functors with classifying spaces $Y_F$ resp. $Y_G$
and universal elements $u_F$ resp. $u_G$. For any natural transformation
$T:F\rightarrow G$ there exists a morphism $f:Y_F\rightarrow Y_G$,
unique up to homotopy, such that the following diagram commutes for
all $X\in \category{C}$:
\begin{center}
  \begin{minipage}[c]{10cm}
    \xymatrix{
      \Hom_{\operatorname{Ho}(\category{C})}(X, Y_F) \ar[r]^{f_\ast}
      \ar[d]_{T_{u_F}(X)} & 
      \Hom_{\operatorname{Ho}(\category{C})}(X, Y_G)
      \ar[d]^{T_{u_G}(X)} \\
      F(X) \ar[r]_{T(X)} & G(X)
    }
  \end{minipage}
\end{center}
\end{corollary}

\begin{proof}
We set $X=Y_F$. Then $T(Y_F)\circ T_{u_F}(Y_F)(\id)$ yields an element 
of $G(X)$. By representability, we have that $T_{u_G}(Y_F)$ is an
isomorphism and hence the element above is of the form
$T_{u_G}(Y_F)(f)$ for a morphism $Y_F\rightarrow Y_G$, which is unique
up to homotopy.
\end{proof}




\subsection{Proof of Homotopy-Continuity}
In this paragraph we will prove that the functor $\fibfun{-}$ from
\prettyref{def:fibfun} is homotopy-continuous. Applying
the Brown representability theorem discussed above, we get classifying
spaces for fibre sequences and universal fibrations.  

First note that $\fibfun{\ast}$ is the singleton set consisting of the
fibre sequence $F\rightarrow F\rightarrow \ast$, settling (HC ii). 

The next serious thing to do is to show (HC i), i.e. that the functor
$\fibfun{-}$ is homotopically meaningful, in the sense that it carries
weak equivalences between simplicial sheaves to isomorphisms of
(pointed) sets. This implies in particular that there is a right
derived functor
$\mathbf{R}\fibfun{-}:\operatorname{Ho}\category{C}^{op}\rightarrow
\category{S}et_\ast$.     

\begin{proposition}
\label{prop:fibfunhoinv}
The functor $\fibfun{-}$ sends weak equivalences of simplicial
sheaves to bijections of pointed sets. 
\end{proposition}

\begin{proof}
Let $f:X\rightarrow Y$ be a weak equivalence, and consider
$f^\ast:\fibfun{Y}\rightarrow \fibfun{X}$. 

To show $f^\ast$ is surjective, let $F\rightarrow
E\stackrel{p}{\rightarrow} X$ be a 
fibre sequence with $p$ a fibration. Consider the following diagram: 
\begin{center}
  \begin{minipage}[c]{10cm}
    \xymatrix{
      E \ar[r]_\simeq^i \ar[d]_p & \tilde{E} \ar[d]^{\tilde{p}} \\
      X \ar[r]_f^\simeq & Y.
    }
  \end{minipage}
\end{center}
Therein, $\tilde{E}$ is obtained by factoring $f\circ p$ into a trivial
cofibration $i$ and a fibration $\tilde{p}$. 
Since both $i$ and $f$ are weak equivalences, this square is a
homotopy pullback. Hence by applying \prettyref{prop:help1} we
get $f^\ast \tilde{p}=p$. 

To see that $f^\ast$ is injective, let $p_1:F\rightarrow
E_1\rightarrow Y$ and $p_2:F\rightarrow E_2\rightarrow Y$ be 
fibre sequences whose pullbacks are equivalent, i.e. $f^\ast
p_1=f^\ast p_2\in \fibfun{X}$. We assume that $p_1$ and $p_2$ are
actually fibrations. By properness and the fact that $f$ is a weak
equivalence, we obtain weak equivalences $E_i\times_Y X\simeq
f^\ast E_i\rightarrow E_i$. Therefore, the fibre sequences
$F\rightarrow f^\ast E_i\rightarrow X$ and $F\rightarrow
E_i\rightarrow Y$ are isomorphic in the homotopy category. Since we
also assumed that $f^\ast p_1=f^\ast p_2$, we have isomorphisms of
fibre sequences in the homotopy category, which are also equivariant:
\begin{displaymath}
E_1\xrightarrow{\phantom{X}\simeq\phantom{X}} 
f^\ast E_1 \xrightarrow{\phantom{X}\simeq\phantom{X}} f^\ast E_2
\xleftarrow{\phantom{X}\simeq\phantom{X}}  E_2.
\end{displaymath} 
Thus $p_1$ and $p_2$ are equivalent fibre sequences. 
\end{proof}

The following propositions will prove the two main parts of
homotopy-continuity of the functor $\fibfun{-}$, namely the
Mayer-Vietoris and the wedge property. 
This is the point where we make essential use of the homotopy
distributivity. This remarkable property 
allows to glue together fibrations defined on a covering of the
base. The outcome will not be a fibration, but we still can determine
the homotopy fibre, and therefore by \prettyref{prop:help1}, we know
what the associated fibre sequence looks like. This is also the key
argument in the work of Allaud \cite{allaud:1966:fibre}, although
there is a lot more to do if one wants to work with homotopy
equivalences of CW-complexes.  

\begin{proposition}[Mayer-Vietoris Axiom]
\label{prop:fibfunmv}
Let $B_1 \stackrel{\iota_1}{\longleftarrow}
B_0\stackrel{\iota_2}{\longrightarrow} B_2$ be a diagram of simplicial 
sheaves. We assume without loss of generality that
$\iota_1:B_0\hookrightarrow B_1$ is in 
fact a cofibration, so that the 
homotopy pushout is given by the point-set pushout:
$B:=B_1\cup^h_{B_0} B_2=B_1\cup_{B_0} B_2$.   

Then the induced morphism
\begin{displaymath}
\fibfun{B_1\cup_{B_0} B_2} \twoheadrightarrow
\fibfun{B_1}\times_{\fibfun{B_0}} \fibfun{B_2} 
\end{displaymath}
is surjective.
\end{proposition}

\begin{proof}
What we have to show is the following: assume given two fibre
sequences $F\rightarrow E_1\rightarrow B_1$ and $F\rightarrow
E_2\rightarrow B_2$, such that the corresponding pullbacks
$\iota_1^\ast E_1$ and $\iota_2^\ast E_2$ are isomorphic fibre
sequences via a given isomorphism 
$\rho:\iota_1^\ast E_1 \stackrel{\cong}{\longrightarrow}\iota_2^\ast
E_2$. Then we have to show that there is a fibre sequence over $B$
inducing them compatibly. 

We will apply Mather's cube theorem from 
\prettyref{cor:mather}. Since we know that the pullbacks of the fibre
sequences are equivalent, the squares in the diagram are homotopy
pullback squares:
\begin{center}
  \begin{minipage}[c]{10cm}
    \xymatrix{
      E_1 \ar[d]_{p_1} & \iota_1^\ast E_1\cong \iota_2^\ast E_2 \ar[l]
      \ar[r] \ar[d] & E_2 \ar[d]^{p_2} \\
      B_1 & B_0 \ar[l]^{\iota_1} \ar[r]_{\iota_2} & B_2.
    }
  \end{minipage}
\end{center}
The homotopy fibres of the vertical arrows are all weakly equivalent
to $F$. Then Mather's cube theorem produces the following fibre
sequence:
\begin{displaymath}
F\rightarrow E_1\cup^h_{E_0} E_2 \stackrel{p}{\longrightarrow} B\cong
B_1\cup^h_{B_0} B_2. 
\end{displaymath}
Actually, we only get the morphism $p$, and know that its
homotopy fibre is $F$. Then we still have to do a
fibrant replacement to really get a fibre
sequence with total space $E:=E_1\cup^h_{E_0}E_2\in\fibfun{B}$. 

What is left to show is that pulling back $E$ to $B_i$ yields
equivalences $\phi:E_1\stackrel{\cong}{\longrightarrow} E$ and
$\psi:E_2\stackrel{\cong}{\longrightarrow} E$ such that over $B_0$ we have
$\phi=\psi\circ \rho$. This also follows from homotopy distributivity:
We assume $p$ has been rectified to a fibration. 
Since the squares 
\begin{center}
  \begin{minipage}[c]{10cm}
    \xymatrix{
      E_i \ar[r] \ar[d]^{p_i} & E \ar[d]^p \\
      B_i \ar[r] & B
    }
  \end{minipage}
\end{center}
are homotopy pullback squares, the map $E_i\rightarrow E$ factors
through a unique weak equivalence from $E_i$ to the point-set pullback
of $E$ along $B_i\rightarrow B$. This provides the equivalences $\phi$
and $\psi$. All we need to show is that the following diagram is
commutative: 
\begin{center}
  \begin{minipage}[c]{10cm}
    \xymatrix{
      &\iota_1^\ast E_1 \ar[d]^\rho
      \ar[dl] \ar[dr] \\
      B_0 &  \iota_2^\ast E_2 \ar[l] \ar[r] \ar[d]^\psi & E \\
      & E\times_BB_0 \ar[ul] \ar[ur]
    }
  \end{minipage}
\end{center}
By the universal property of the pullback, this implies that
$\psi\circ\rho=\phi$, since the morphism $\iota_1^\ast E_1\rightarrow
E\times_BB_0$ is by definition $\phi$.

The upper left triangle commutes, since $\rho$ was defined over $B_0$.
The lower triangles commute because $\psi$ was defined using the
universal property of the pullback $E\times_B B_0$. The upper right
triangle commutes because $E$ was defined as the glueing of
$\iota_1^\ast E_1$ and $\iota_2^\ast E_2$ along $\rho$. Therefore, we 
get that $\phi=\psi\circ \rho$. In particular, the image of
$E\in\fibfun{B}$ in $\fibfun{B_1}\times_{\fibfun{B_0}}\fibfun{B_2}$ is
exactly the class of $(E_1,E_2,\rho)$ we started with.

Finally, note that the result of glueing two locally trivial fibre
sequences in a homotopy colimit produces again a locally trivial fibre
sequence. The canonical homotopy colimit decomposition reduces this
assertion to the case of homotopy colimits of simplicial dimension
zero representable sheaves, where it is obvious.
\end{proof}



\begin{proposition}[Wedge Axiom]
\label{prop:fibfunwedge}
Let $B_\alpha$ with $\alpha\in I$ be a set of pointed simplicial
sheaves. Then  
\begin{displaymath}
\fibfun{\bigvee_\alpha B_\alpha}\rightarrow
\prod_\alpha\fibfun{B_\alpha}: E\mapsto \iota_\alpha^\ast(E) 
\end{displaymath}
is a bijection, where $\iota_\alpha:B_\alpha\hookrightarrow
\bigvee_\alpha B_\alpha$ denotes the inclusion. In particular, the
collection $\fibfun{\bigvee_\alpha B_\alpha}$ is a set if
$\fibfun{B_\alpha}$ is a set for every $\alpha\in I$.
\end{proposition}

\begin{proof}
Define $E$ via the following homotopy pushout, where the maps
$F\rightarrow E_\alpha$ are the ones from the definition of fibre
sequence: 
\begin{center}
  \begin{minipage}[c]{10cm}
    \xymatrix{
      \bigvee_\alpha F \ar[r] \ar[d] & \bigvee_\alpha E_\alpha \ar[d] \\
      F \ar[r] & E
    }
  \end{minipage}
\end{center}
By homotopy distributivity we find 
that the following squares are homotopy pullbacks for all $\alpha$:
\begin{center}
  \begin{minipage}[c]{10cm}
    \xymatrix{
      E_\alpha \ar[r] \ar[d] & E \ar[d] \\
      B_\alpha \ar[r] & \bigvee_\alpha B_\alpha.
    }
  \end{minipage}
\end{center}
Therefore, the homotopy fibre of $\bigvee_\alpha p_\alpha$ is also
$F$. Rectifying it to a fibre sequence, we get $\bigvee_\alpha
E_\alpha\in\fibfun{\bigvee_\alpha B_\alpha}$. This proves
surjectivity. 

Now assume given two fibre sequences $E_1$ and $E_2$ over
$\bigvee_\alpha B_\alpha$. We first prove the weak equivalence
$\bigvee_\alpha \iota_\alpha^\ast E\simeq E$ for any fibre sequence
$E\rightarrow \bigvee_\alpha B_\alpha$. This follows from
distributivity in categories of simplicial sheaves, since $E$ is
isomorphic to the colimit of $\iota_\alpha^\ast E_\alpha$. Then one
can either use that the wedge is already the homotopy direct product,
or again appeal to the homotopy distributivity. 
Then we have the sequence of weak equivalences:
\begin{displaymath}
E_1 \mathrel{\simeq} \bigvee_\alpha \iota_\alpha^\ast E_1
\mathrel{\simeq} \bigvee_\alpha 
\iota_\alpha^\ast E_2\mathrel{\simeq} E_2.
\end{displaymath}
The middle weak equivalence follows from the fact \cite[Lemma 
  13.(3)]{jardine:1996:boolean} that for a set-indexed collection of
weak equivalences $f_\alpha:B_\alpha\rightarrow Y_\alpha$, the morphism
$\bigvee_\alpha f_\alpha$ is also a weak equivalence and the
assumption that $\prod \iota_\alpha^\ast E_1=\prod\iota_\alpha^\ast
E_2$ in $\prod_\alpha\fibfun{B_\alpha}$. 

The set theory statement is then clear, since a set-indexed product of
sets is again a set. 
\end{proof}

\newcommand{\classify}[1]{\ensuremath{B^f{#1}}}

\begin{theorem}
\label{thm:classify}
Assume  the model category $\simplicial{\topos{T}}$ is
compactly generated. The functor $\fibfun{-}$ which associates to each
simplicial sheaf $X$ the set of fibre sequences over $X$ with fibre
$F$ is homotopy continuous and therefore representable by a space
$\classify{F}$. The space $\classify{F}$ is unique up to weak equivalence. 

The universal element $u_F\in\fibfun{\classify{F}}$ corresponds to the
universal fibre sequence of simplicial sheaves with fibre $F$: 
\begin{displaymath}
F\rightarrow E^fF\rightarrow \classify{F}
\end{displaymath}
\end{theorem}

\begin{remark}
\begin{enumerate}[(i)]
\item
We use the notation $\classify{F}$ to distinguish from other possible
classifying spaces we will discuss later on. 
\item There should be a simplicial functor associating to each space
  $B$ the nerve of the category of fibre sequences over this
  space -- one has to circumvent the obvious set-theoretical
  difficulty in the construction. It seems likely that the other Brown
  representability theorem \cite[Theorem 12]{jardine:2006:brown} can
  then be applied to this functor. This would allow to remove the
  compact generation hypothesis in the above theorem. Anyway, the next
  section will show that classification of non-rooted fibrations is
  always possible. 
\end{enumerate}
\end{remark}

Using homotopy distributivity once again, we can construct
change-of-fibre natural transformations, which via Brown
representability for 
morphisms give rise to morphisms between the corresponding classifying
spaces. We state the result without giving the obvious proof.

\begin{proposition}
\label{prop:fibchange}
Assume the model category $\simplicial{\topos{T}}$ is compactly
generated. 
For any morphism $f:F\rightarrow F'$ of simplicial sheaves, there
is a natural transformation $\fibfun{-}\rightarrow
\mathcal{H}^{\operatorname{pt}}(-,F')$. By Brown representability this
is representable by a morphism $\classify{F}\rightarrow
\classify{F'}$.   
\end{proposition}

Similarly, it is possible to generalize operations on fibrations,
cf. \cite[Proposition 1.43]{rudyak:1998:cobordism} or
\cite{may:1980:fibrewise}, to the simplicial sheaf setting.

\begin{corollary}
\label{cor:fibsmash}
Assume the model category $\simplicial{\topos{T}}$ is compactly
generated. 
There are morphisms of classifying spaces associated to fibrewise
smash 
\begin{displaymath}
\classify{(\wedge)}:\classify{F_1}\times\classify{F_2}\rightarrow
\classify{(F_1\wedge F_2)},
\end{displaymath}
and fibrewise suspension $\classify{F}\rightarrow\classify{\Sigma
  F}$. 
\end{corollary}

\section{Second Variant: Bar Construction}
\label{sec:two}

\renewcommand{\fibfun}[1]{\ensuremath{\mathcal{H}({#1},F)}}

In this section, we explain the second approach to the construction of
classifying spaces of fibre sequences. Again, this approach is a
direct generalization of results that are known for topological spaces
resp. simplicial sets. 
The first result in this direction is the work of Stasheff
\cite{stasheff:1963} which proves that fibrations over CW-complexes
with a given finite CW-complex as fibre can be classified by homotopy
classes of maps into some CW-complex. In fact the classifying space is
the classifying space of the topological monoid of homotopy
self-equivalences of the fibre. 
The main idea in this approach is the construction of an associated
principal bundle for a fibration. This associates to a fibration
$p:E\rightarrow B$ a new fibration
$\operatorname{Prin(p)}:\operatorname{Prin(E)}\rightarrow B$, whose
fibre have the homotopy type of the topological monoid of homotopy
self-equivalences. 

A vast generalization of this can be found in
\cite{may:1975:fibrations}. There, the double bar construction is used
to construct the classifying spaces for fibrations with given
fibre. Moreover, the notion of a category of fibres allows to classify
fibrations with global structures. Again the main point in proving
the classification theorem is a principalization construction which
associates to a fibration a principal bundle. 

These results can be translated to simplicial sheaves. One problem
that appears in this setting is that principalization does not work a
priori. The way around this is again the restriction to fibre
sequences which are locally trivial in a given topology. These
trivializations indeed allow to translate the principalization
construction. 

In the case of CW-complexes, the local triviality condition is
no restriction at all: every point has a contractible
neighbourhood $U$. For a fibration over such a contractible
neighbourhood, the inclusion of the fibre is a weak equivalence (in
fact a homotopy 
equivalence if all the spaces in sight are CW-complexes). This means
that there is a morphism (over $U$) $E\rightarrow F\times U$ which is a
weak equivalence. This provides the local trivializations in the case
of CW-complexes. In fact, it allows to construct a morphism from the
associated \v Cech-complex of a fine enough covering to the
classifying space of the monoid of self-equivalences -- for each
intersection $U_i\cap U_j$ of contractible neighbourhoods there is a
morphism $U_i\cap U_j\rightarrow \haut(F)$ corresponding to the
composition of the two trivializations over $U_i$ resp. $U_j$. The
cocycle condition is not satisfied on the nose, but up to
homotopy. Therefore, one obtains a morphism $U_i\cap U_j\cap U_k\times
I\rightarrow \haut(X)$ etc. Since the realization of the \v Cech
complex is homotopy equivalent to the CW-complex we started with, we
obtain a map in the homotopy category $B\rightarrow B\haut(F)$. This
is a slightly souped up version of the principalization construction,
which also works in the simplicial sheaf setting. Hopefully, it has
become clear with the above discussion that the local triviality
condition on the fibre sequences comes in rather naturally in the bar
construction approach.

\subsection{Fibre Sequence Functor}

We now define the functor which will be represented. In the case of
the bar construction, this functor is the unpointed analogue of the
one defined in \prettyref{def:fibfun}. It associates to an unpointed
simplicial sheaf $B$ the set of all locally trivial fibre sequences
over $B$ with fibre $F$. Therefore, it does not fix an equivalence
between $F$ and $p^{-1}(\ast)$. 

\begin{definition}
Recall the definition of locally trivial morphism with fibre $F$. 
Two locally trivial morphisms $p_1:E_1\rightarrow B$ and
$p_2:E_2\rightarrow B$ with fibre $F$ are said to be \emph{equivalent}
if there is a diagram in the homotopy category
\begin{center}
  \begin{minipage}[c]{10cm}
    \xymatrix{
      E_1\ar[rr]^\alpha \ar[d]_{p_1} && E_2 \ar[d]^{p_2}\\
      B\ar[rr]_\id &&B
    }
  \end{minipage}
\end{center}
where $\alpha$ is an isomorphism. We denote by $\fibfun{X}$ the set of
locally trivial morphisms over $X$ with fibre $F$ modulo the above
equivalence relation. 
\end{definition}

\begin{remark}
\begin{enumerate}[(i)]
\item Assuming that the $p_i$ are fibrations, we can use the homotopy
  lifting to obtain a morphism which respects fibres. So if $p_1$ and
  $p_2$ are equivalent, then there is an equivalence which respects
  the fibres.
\item The analogue of   \prettyref{prop:fibfunset} can be proved in
  complete analogy, we  omit the proof.
\item In case $X$ is actually pointed, we can obtain the set
  $\fibfun{X}$ by taking fibre sequences over $X$ modulo the
  equivalence relation given by ladder diagrams in the homotopy category
\begin{center}
  \begin{minipage}[c]{10cm}
    \xymatrix{
      F \ar[d]_\beta \ar[r] & E_1 \ar[d]^\alpha \ar[r] & X
      \ar[d]^{\id}\\ 
      F \ar[r] & E_2 \ar[r] & X,
    }
  \end{minipage}
\end{center}
where $\alpha$ and $\beta$ should be isomorphisms. 
\end{enumerate}
\end{remark}

\subsection{Remarks on Categories of Fibres}
In the following, we will not work in the full generality of
categories of fibres. Rather we will only consider the fibre sequences
which are locally trivial from \prettyref{def:fibfun}. However, we
want to make a few remarks on the possible definition of category of
fibres for simplicial sheaves. 

The original definition of categories of fibres can be found in
\cite{may:1975:fibrations}. 
A definition of categories of fibres in equivariant topology has been
given Waner \cite[Definition 1.1.1]{waner:1980:classify} resp. French   
\cite[Definition 3.1]{french:2003:jhom} for equivariant homotopy
theory. These definitions readily generalize to simplicial
sheaves. One should however note that equivariant topology is a
presheaf situation without a Grothendieck topology (at least in the
case of finite groups) -- in the full generalization it is therefore
necessary to include a localization condition. 

\begin{definition}[Category of Fibres]
Let $T$ be a site. A \emph{category of fibres} is a subcategory
$\mathcal{F}_T$ of the following category: 
\begin{itemize}
\item Objects are morphisms
$p:X\rightarrow U$ of simplicial sheaves, where $U$ is the constant
simplicial sheaf for a representable $U\in T$.
\item Morphisms are commutative diagrams
\begin{center}
  \begin{minipage}[c]{10cm}
    \xymatrix{
      X \ar[r] \ar[d] & X' \ar[d] \\
      U \ar[r] & U'
    }
  \end{minipage}
\end{center}
\end{itemize}
Additionally, we require that
\begin{enumerate}[(CFi)]
\item
The map $X\rightarrow U$ is required to be locally trivial in the
$T$-topology.
\item  For a morphism 
\begin{center}
  \begin{minipage}[c]{10cm}
    \xymatrix{
      X \ar[r] \ar[d] & X' \ar[d] \\
      U \ar[r] & U'
    }
  \end{minipage}
\end{center}
there is a $T$-covering $\bigsqcup U_i\rightarrow U'$ such that the
induced morphisms $X\times_{U'}U_i\rightarrow X'\times_{U'}U_i$
are weak equivalences of simplicial sheaves.
\end{enumerate}

As in the equivariant definitions of categories of fibres one wants to
have a simplicial sheaf $F$ which serves as a model for the fibres:
a corresponding category should contain at least the obvious objects
$p_2:F\times U\rightarrow U$ for $U\in T$, together with the obvious 
morphisms 
\begin{center}
  \begin{minipage}[c]{10cm}
    \xymatrix{
      F\times U_1 \ar[r]^{\id\times f} \ar[d]_{p_2} & F\times U_2
      \ar[d]^{p_2} \\ 
      U_1 \ar[r]_f & U_2
    }
  \end{minipage}
\end{center}
induced from $f:U_1\rightarrow U_2$ in $T$. 
\end{definition}

The notion of $\Gamma$-completeness which appears in the cited works
on categories of fibres basically state that the category of fibres
should be closed under fibrant replacements. This is needed since some
constructions (like glueing) yield quasi-fibrations instead of
fibrations, and one would like to replace them by fibrations without
losing the property that the fibres are elements in the category of
fibres. 

The basic definitions and results concerning categories of fibres and
their principalizations can then be translated from
e.g. \cite{french:2003:jhom}. As said before, we will only consider
locally trivial fibre sequences with given fibre, as defined in
\prettyref{def:fibfun}. It is easy to check that this definition can
be formulated as a special case of a category of fibres.

\subsection{Homotopy Self-Equivalences}
Most important for our studies in the sequel will be the simplicial
monoid of homotopy self-equivalences of a simplicial sheaf. This is
the obvious generalization of the homotopy self-equivalences of a
simplicial set. 

We first recall the definition of homotopy self-equivalences of
simplicial sets. For more details on function complexes of simplicial
sets, see \cite[Section I.5]{goerss:jardine:1999:simplicial}. Function
complexes in general model categories are constructed in
\cite{dwyer:kan:1980:hammock}. A general discussion about what is
known for the monoids of homotopy self-equivalences can be found in
\cite{rutter:1997}. 

\begin{definition}
\label{def:hautsimp}
Let $X$ be a fibrant simplicial set. Then there is a simplicial set
$\inthom(X,X)$ whose set of $n$-simplices is given by
\begin{displaymath}
\inthom(X,X)_n=\hom_{\simplicial{\set}}(X\times\Delta^n,X).
\end{displaymath}
This is a special case of function complexes of simplicial sets,
cf. \cite{goerss:jardine:1999:simplicial}. 
By standard facts on function complexes, there is a fibration 
\begin{displaymath}
\inthom(X,X)\rightarrow \inthom(\ast,X)\simeq X,
\end{displaymath}
therefore $\inthom(X,X)$ is also a fibrant simplicial set.

The monoid structure can be described as follows: for two maps
$f,g:\Delta^n\times X\rightarrow X$, their composition $f\circ g$ in
the monoid $\inthom(X,X)_n$ is given by 
\begin{displaymath}
f\circ g:\Delta^n\times X\stackrel{D\times \id}{\longrightarrow}
\Delta^n\times\Delta^n\times X\stackrel{\id\times g}{\longrightarrow}
\Delta^n\times X \stackrel{f}{\longrightarrow} X,
\end{displaymath}
where $D:\Delta^n\rightarrow \Delta^n\times\Delta^n$ is the diagonal
morphism on the standard $n$-simplex $\Delta^n$.

It is obvious that the simplicial subset of morphisms $X\rightarrow X$
which are weak equivalences is in fact a simplicial submonoid. 
The resulting monoid of homotopy self-equivalences is denoted by
$\haut(X)$. 

Note that this monoid is group-like since $X$ is cofibrant and
fibrant. In this case, a weak equivalence $f:X\rightarrow X$ is a
homotopy equivalence and therefore its class in $\pi_0 \inthom(X,X)$
has an inverse.  
\end{definition}

The general definition of homotopy self-equivalences in general model
category was given by Dwyer and Kan in
\cite{dwyer:kan:1980:hammock}. Their construction yields for an object
$X$ in a model category $\mathcal{C}$ a function complex
$\hom(X,X)$ which is a simplicial set. For simplicial sheaves, we
can additionally use the internal Hom to obtain a simplicial sheaf of
monoids of homotopy self-equivalences. 
It is explained in \cite[Remark 1.1.7, Lemma
1.1.8]{morel:voevodsky:1999:a1} that the category of simplicial
sheaves has internal hom-objects.

\begin{definition}
Let $T$ be a site, and let $X$ be a fibrant simplicial sheaf. We
define the sheaf of self-homotopy equivalences, which is a simplicial
sheaf of monoids. 
By \prettyref{thm:jardine}, the simplicial sheaves on $T$ form a
simplicial model category, hence for any two simplicial sheaves $X,Y$
there is a simplicial set, the function complex $\Hom(X,Y)$, whose
$n$-simplices are given by 
\begin{displaymath}
\Hom_{\simplicial{\topos{T}}}(X\times \Delta^n,Y).
\end{displaymath}

In particular, we have a contravariant functor 
\begin{displaymath}
T^{\operatorname{op}}\rightarrow \simplicial{\set}:(U\in T)\mapsto
\Hom_{\simplicial{\topos{T}}}(X\times U,X).
\end{displaymath} 
This functor is representable by a simplicial sheaf which we again
denote by 
$$\inthom_{\simplicial{\topos{T}}}(X,X). $$

We can define a subpresheaf by taking for $U\in T$ the subset of those
morphisms $\Hom_{\simplicial{\topos{T}}}(X\times U,X\times U)$ which
are weak equivalences of simplicial sheaves in
$\simplicial{\topos{T}}$. Note that this is indeed a sheaf because
weak equivalences are defined locally: given a covering $\bigsqcup
U_i\rightarrow U$ and weak equivalences $f_i:X\times U_i\rightarrow
X\times U_i$ which agree on the intersections, there is morphism
$f:U\rightarrow U$ which is a weak equivalence if all the $f_i$ are
weak equivalences.  

The resulting simplicial sheaf of monoids will be denoted by
$\haut(X)$. The monoid structure is again given by composition as in
\prettyref{def:hautsimp}. 

Note that the simplicial sheaf of monoids $\haut(X)$ is fibrant if $X$
is. This is a consequence of the simplicial model structure on
simplicial sheaves, cf. \cite[Proposition
II.3.2]{goerss:jardine:1999:simplicial}:
the morphisms $\inthom(X,X)\rightarrow \inthom(\ast,X)$ and
$\inthom(\ast,X)\rightarrow \inthom(\ast,\ast)\cong \ast$ induced
from the morphism $X\rightarrow \ast$ are fibrations if $X$ is
fibrant.
\end{definition}

\begin{lemma}
Let $X$ be a fibrant simplicial sheaf on the site $T$. Then $X$ is a
left $\haut(X)$ space, i.e. there is an
action 
\begin{displaymath}
\haut(X)\times X\rightarrow X.
\end{displaymath}
\end{lemma}

Note that if $X$ is fibrant, then a morphism $X\rightarrow X$ is a
weak equivalence if and only if the morphism induced on sections
$f(U):X(U)\rightarrow X(U)$ is a weak equivalence of simplicial sets
for all $U\in T$, cf. \cite[Lemma
I.1.10]{morel:voevodsky:1999:a1}. Therefore, $\haut(X)(U)$ acts on
$X(U)$ via homotopy self-equivalences of simplicial sets.
Note also that the action is really an action in
$\simplicial{\topos{T}}$, not just an action in the homotopy
category. 

\subsection{The Bar Construction}
We repeat the definition and basic properties of the 
bar construction following \cite{may:1975:fibrations}. Again the
setting changes from topological spaces to simplicial sheaves
without major complications, cf. also \cite[Example
4.1.11]{morel:voevodsky:1999:a1}.  

\begin{definition}[Two-sided geometric bar construction]
Let $G$ be a simplicial sheaf of monoids on the site $T$. We assume
that the inclusion of the identity $e\rightarrow G$ is a
cofibration. For the injective model structure, this is no problem
because every monomorphism is a cofibration. 
Let $X$ and $Y$ be simplicial sheaves, such that $X$ has a left
$G$-action and $Y$ has a right $G$-action.

Then there is a bisimplicial sheaf 
\begin{displaymath}
B_{n,m}(Y,G,X)=(Y\times G^n\times X)_m.
\end{displaymath}
For an object $U$ of the site $T$, we have
\begin{displaymath}
B_{n,m}(Y,G,X)(U)=B_{n,m}(Y(U),G(U),X(U)),
\end{displaymath}
 and functoriality of the
bar construction for simplicial sets provides the restriction maps to
turn this into a simplicial sheaf. Similarly, the face and
degeneracy maps are functorial, and hence provide $B_{n,m}(Y,G,X)$
above with the structure of bisimplicial sheaf.
The diagonal $B_{n,n}(Y,G,X)$ is a simplicial sheaf, which we will
denote $B(Y,G,X)$.
\end{definition}

The classifying spaces for simplicial sheaves of monoids can then
be obtained as $BG=B(\ast,G,\ast)$, and the universal $G$-bundle is
given by the obvious functoriality:
\begin{displaymath}
EG=B(\ast,G,G)\rightarrow B(\ast,G,\ast)=BG.
\end{displaymath}

The topology enters via a fibrant replacement: for a simplicial
sheaf $X$, any morphism $X\rightarrow BG$ in the homotopy category
can be represented up to homotopy by a morphism $X'\rightarrow BG$ for
some suitable trivial local fibration $X'\rightarrow X$. The notion of
trivial local fibration depends on the topology, as a trivial local
fibration is a morphism of simplicial sheaves which induces a
trivial Kan fibration of simplicial sets on the stalks.
The fibrant replacement may change the global sections of $B(Y,G,X)$,
but it does not change the homotopy types of the stalks, which
therefore can be described as the bar constructions for the simplicial
sets $p^\ast Y$, $p^\ast G$ and $p^\ast X$.

The following properties of the bar construction for simplicial
sheaves are direct consequences of the corresponding properties for
simplicial sets resp. topological spaces, cf. \cite[Section
7]{may:1975:fibrations}: 

\begin{proposition}
\begin{enumerate}[(i)]
\item The space $B(Y,G,X)$ is $n$-connected provided $G$ is
  $(n-1)$-connected and $X$ and $Y$ are $n$-connected. 
\item If $f_1:Y\rightarrow Y'$, $f_2:G\rightarrow G'$ and
  $f_3:X\rightarrow X'$ are weak equivalences of simplicial
  sheaves, then the morphism $f:B(Y,G,X)\rightarrow
  B(Y',G',X')$ is a weak equivalence. 
\item For $(Y,G,X)$ and $(Y',G',X')$ the projections define a natural
  weak equivalence 
\begin{displaymath}
B(Y\times Y',G\times G',X\times X')\rightarrow B(Y,G,X)\times B(Y',G',X').
\end{displaymath}
\item Let $f:H\rightarrow G$ be a morphism of simplicial sheaves of
  monoids, and let $k:Z\rightarrow Y$ be an equivariant morphism of
  right $G$-spaces. Then the following diagrams are pullbacks:
\begin{center}
  \begin{minipage}[c]{10cm}
    \xymatrix{
      B(Z,H,X) \ar[d]_p \ar[rr]^{B(k,f,\operatorname{id})} && 
      B(Y,G,X) \ar[d]^p \\
      B(Z,H,\ast) \ar[rr]_{B(k,f,\operatorname{id})} &&
      B(Y,G,\ast).
    }
  \end{minipage}
\end{center}
\begin{center}
  \begin{minipage}[c]{10cm}
    \xymatrix{
      B(Y,G,X) \ar[rr]^q \ar[d]_p && B(\ast,G,X)\ar[d]^p \\
      B(Y,G,\ast) \ar[rr]_q && BG
    }
  \end{minipage}
\end{center}
\end{enumerate}
\end{proposition}

\begin{proof}
For (i), note that $n$-connectedness means that the homotopy group
\emph{sheaves} $\pi_i(B(Y,G,X))$  are trivial for $i\leq n$. In
particular, this does not imply that the simplicial sets $B(Y,G,X)(U)$
are $n$-connected for any $U\in T$. 

All four statements are of a local nature, i.e. can be checked on
stalks. The corresponding statements for topological spaces 
are Propositions 7.1, 7.3, 7.4 and 7.8 of \cite{may:1975:fibrations}.
\end{proof}

The following result is a version of 
\cite[Theorem 7.6, Proposition 7.9]{may:1975:fibrations} for
simplicial sheaves. It provides necessary fibre sequences for the
proof of the classification theorem. Note that for any simplicial
sheaf of monoids $M$, the monoid operation induces a monoid operation
on the sheaf $\pi_0 M$. We say that $M$ is \emph{grouplike} if this operation
turns $\pi_0M$ into a sheaf of groups.

\begin{theorem}
\label{thm:mayfib}
If $G$ is grouplike, there are fibre sequences of simplicial
sheaves 
\begin{enumerate}[(i)]
\item $X\rightarrow B(Y,G,X)\rightarrow B(Y,G,\ast)$,
\item $Y\rightarrow B(Y,G,X)\rightarrow B(\ast, G,X)$, and 
\item $G\rightarrow Y\rightarrow B(Y,G,\ast).$
\end{enumerate}
\end{theorem}

\begin{proof}
The corresponding statements for simplicial sets resp. topological
spaces can be found as \cite[Theorem 7.6, Proposition
7.9]{may:1975:fibrations}. The
corresponding statements are true for simplicial sheaves by
\prettyref{prop:pullbackloc}: 
everything that locally (i.e. on stalks) looks like a fibre sequence,
really is a fibre sequence.
\end{proof}

\subsection{The Classification Theorem}

Now we come to the proof of the classification theorem. The
classifying space is given by the bar construction $B(\ast,\haut(F),\ast)$ and
the universal fibre sequence is 
\begin{displaymath}
F\rightarrow B(\ast,\haut(F),F)\rightarrow B(\ast,\haut(F),\ast).
\end{displaymath}
It follows from the previous \prettyref{thm:mayfib} that this is
indeed a fibre sequence of simplicial sheaves. 

The following is a version of May's classification result
\cite[Theorem 9.2]{may:1975:fibrations} for simplicial sheaves.
The argument in the topological case can be found in
\cite{stasheff:1974,stasheff:wirth}. 

\begin{theorem}
\label{thm:classmay}
Let $T$ be a site, $F$ be a fibrant simplicial sheaf on $T$. Then
there is a natural isomorphism of functors 
\begin{displaymath}
\mathcal{H}(X,F)\cong [X,B(\ast,\haut(F),\ast)],
\end{displaymath}
where the right-hand side denotes the set of morphisms
\begin{displaymath}
X\rightarrow B(\ast,\haut(F),\ast)
\end{displaymath}
in the homotopy category. 
\end{theorem}

\begin{proof}
The universal fibre sequence is 
\begin{displaymath}
F\rightarrow B(\ast,\haut(F),F)\rightarrow B(\ast,\haut(F),\ast).
\end{displaymath}
We can replace this by an honest fibration of fibrant simplicial
sheaves whose fibre is weakly equivalent to $F$. This can be viewed as
an element of 
$$\mathcal{H}(B(\ast,\haut(F),\ast),F)$$
which we denote by $\pi$.

(i)
Now we define a natural transformation
\begin{displaymath}
\Psi:[X,B(\ast,\haut(F),\ast)]\rightarrow \mathcal{H}(X,F):f\mapsto
f^\ast \pi
\end{displaymath}
This is well-defined and natural by \prettyref{prop:help1}.

(ii)
In the other direction, we define 
\begin{displaymath}
\Phi:\mathcal{H}(X,F)\rightarrow [X,B(\ast,\haut(F),\ast)]
\end{displaymath}
via the following principalization construction. 
Let $F\rightarrow E\rightarrow X$ be a fibre sequence in
$\mathcal{H}(X,F)$. By assumption, this is locally trivial, i.e. 
there exists a covering $\bigsqcup U_i\rightarrow X$ such that 
$E\times_X U_i\simeq F\times U_i$. 

By composition of the two trivializations for $U_i, U_j$, we obtain a
weak equivalence over $U_i\times_XU_J$: 
\begin{displaymath}
\phi_{ij}:F\times (U_i\times_XU_j) \rightarrow F\times (U_i\times_XU_j),
\end{displaymath}
which corresponds to a morphism $U_i\times_X U_j \rightarrow \haut(F)$.

Then there is a diagram of weak equivalences 
\begin{center}
  \begin{minipage}[c]{10cm}
    \xymatrix{
      F\times (U_i\times_XU_j\times_X U_k) \ar[rr]^{\phi_{ij}} 
      \ar[drr]_{\phi_{ik}} && 
      F\times (U_i\times_XU_j\times_X U_k) \ar[d]^{\phi_{jk}} \\
      &&F\times (U_i\times_XU_j\times_X U_k).
    }
  \end{minipage}
\end{center}
This diagram is not commutative but commutative up to homotopy, hence
gives rise to a morphism $U_i\times_X U_j\times_XU_k\times\Delta^1
\rightarrow \haut(F)$.

In the usual way, we obtain a $T$-hypercovering $U_\bullet\rightarrow
X$ and a morphism of simplicial sheaves $U_\bullet\rightarrow
B(\ast,\haut(F),\ast)$. This is indeed a morphism 
$$X\rightarrow B(\ast,\haut(F),\ast)$$ 
in the homotopy category because hypercoverings are locally trivial
fibrations.  

This is well-defined, since the category of hypercoverings is
filtered. For any two hypercoverings $U_\bullet$ and $U_\bullet'$ and
maps $U_\bullet\rightarrow B(\ast,\haut(F),\ast)$, there is a
refinement $V_\bullet$ of both $U_\bullet$ and $U'_\bullet$ and a
homotopy between the two corresponding maps $V_\bullet\rightarrow
U_\bullet \rightarrow B(\ast,\haut(F),\ast)$ and $V_\bullet\rightarrow
U'_\bullet \rightarrow B(\ast,\haut(F),\ast)$. For the basic
assertions concerning hypercovers, see \cite{friedlander:1982}. 

(iii) The composition $\Psi\circ\Phi$ is the identity on
$\mathcal{H}(X,F)$. This means that a fibre sequence $F\rightarrow
E\rightarrow X$ is equivalent to $f^\ast\pi$ for $f:X\rightarrow
B(\ast,\haut(F),\ast)$ the morphism constructed in (ii). By
\prettyref{prop:help1}, it suffices to check this for the
hypercovering $U_\bullet$. But since the fibre sequence over
$U_\bullet$ is explicitly trivialized, the principalization consists
of replacing $F\times U_i$ with $\haut(F)\times U_i$. The pullback of
the universal fibre sequence along $F$ replaces $\haut(F)$ again by
$F$. Hence $\Psi\circ\Phi$ is the identity.

(iv) The composition $\Phi\circ \Psi$ is the identity on
$[-,B(\ast,\haut(F),\ast)]$. Any map in the homotopy category from $X$
to $B(\ast,\haut(F),\ast)$ can be represented by a hypercovering
$U_\bullet\rightarrow X$ and a morphism $U_\bullet\rightarrow 
B(\ast,\haut(F),\ast)$. This hypercovering trivializes the
corresponding fibre sequence, and the associated principal
$\haut(F)$-bundle is obtained by replacing $F$ by $\haut(F)$ as in
(ii). The resulting map $f:U_\bullet\rightarrow B(\ast,\haut(F),\ast)$
is the map we started with.
\end{proof}

\begin{remark}
In case $X$ is pointed, the relation between the classifying spaces
constructed in \prettyref{sec:one} and \prettyref{sec:two} is as
follows: the global sections of the sheaf $\pi_0\haut(F)$ act on the
set $\mathcal{H}^{\operatorname{pt}}(X,F)$, and the quotient modulo
this action is $\fibfun{X}$. We can not state a more general result as
there are simplicial presheaves which can not be pointed because they
do not have global sections. 
\end{remark}


\begin{thebibliography}{Wen10}


\bibitem[All66]{allaud:1966:fibre}
G. Allaud. On the Classification of Fiber Spaces. Math. Z. 92 (1966),
110--125. 

\bibitem[Bro62]{brown:1962:representability}
E.H. Brown, jr. Cohomology theories. Ann. Math. 75 (1962), 467--484. 


\bibitem[DF96]{farjoun:1996:cellular}
E. Dror Farjoun. Cellular spaces, null spaces and homotopy
localization. Lecture Notes in Mathematics 1622, Springer (1996).

\bibitem[DK80]{dwyer:kan:1980:hammock}
W.G. Dwyer and D.M. Kan. Function Complexes in Homotopical
Algebra. Topology 19 (1980), 427--440.

\bibitem[DK84]{dwyer:kan:1984}
W.G. Dwyer and D.M. Kan.
A classification theorem for diagrams of simplicial sets.
Topology 23 (1984), no. 2, 139--155. 

\bibitem[DL59]{dold:lashof}
A. Dold and R. Lashof. 
Principal quasi-fibrations and fibre homotopy equivalence of bundles.
Illinois J. Math. 3 (1959), 285--305. 

\bibitem[Dol63]{dold:1963:fib}
A. Dold. Partitions of unity in the theory of fibrations.
Ann. Math. (2) 78 (1963), 223--255.

\bibitem[Dol66]{dold:1966}
A. Dold. Halbexakte Homotopiefunktoren. Lecture Notes in Mathematics
12. Springer-Verlag (1966).

\bibitem[DT58]{dold:thom}
A. Dold and R. Thom. 
Quasifaserungen und unendliche symmetrische Produkte.
Ann. of Math. (2) 67 (1958), 239--281. 

\bibitem[Dug01]{dugger:2001:combinatorial}
D. Dugger. Combinatorial model categories have presentations.
Adv. Math. 164 (2001), 177--201. 

\bibitem[Fre03]{french:2003:jhom}
C.P. French. The equivariant $J$-homomorphism. Homology Homotopy
Appl. 5 (2003), 161--212. 

\bibitem[Fri82]{friedlander:1982}
E.M. Friedlander. {\'E}tale homotopy of simplicial schemes. 
Annals of Mathematics Studies, 104. Princeton University Press, 1982.

\bibitem[Gan65]{ganea:1965:suspension}
T. Ganea. A generalization of homology and homotopy
suspension. Comment. Math. Helv. 39 (1965), 295--322.

\bibitem[GJ99]{goerss:jardine:1999:simplicial}
P.G. Goerss and J.F. Jardine.
Simplicial Homotopy Theory. Progress in Mathematics 174. Birkh\"auser
(1999). 

\bibitem[Hel81]{heller:1981:brown}
A. Heller. On the representability of homotopy functors. 
J. London Math. Soc.(2) 23 (1981), 551--562.

\bibitem[Hir03]{hirschhorn:2003:modelcats}
P.S. Hirschhorn. Model categories and their localizations. 
Mathematical Surveys and Monographs 99, American Mathematical Society
(2003). 

\bibitem[Hor06]{hornbostel:2006}
J. Hornbostel. Localizations in motivic homotopy theory.
Math. Proc. Cambridge Philos. Soc. 140 (2006), 95--114. 

\bibitem[Hov98]{hovey:1998:modelcats}
M.A. Hovey. Model categories. 
Mathematical Surveys and Monographs 63, American Mathematical Society
(1998). 

\bibitem[Jar87]{jardine:1987:simplicial}
J.F. Jardine. Simplicial presheaves. J. Pure Appl. Algebra 47 (1987),
35--87. 

\bibitem[Jar96]{jardine:1996:boolean}
J.F. Jardine. Boolean localization, in practice. Doc. Math. 1(1996),
245--275. 

\bibitem[Jar11]{jardine:2006:brown}
J.F. Jardine. Representability theorems for presheaves of
spectra. J. Pure Appl. Algebra 215 (2011), 77--88.



\bibitem[Mac98]{maclane:1998:categories}
S. Mac Lane. Categories for the working mathematician. Graduate Texts
in Mathematics 5. Springer (1998).

\bibitem[Mat76]{mather:1976:hopullback}
M. Mather. Pullbacks in homotopy theory. Canad. J. Math. 28 (1976),
225--263. 

\bibitem[May75]{may:1975:fibrations}
J.P. May. Classifying Spaces and Fibrations. Mem.
Amer. Math. Soc. 155 (1975).

\bibitem[May80]{may:1980:fibrewise}
J.P. May. Fibrewise localization and completion. 
Trans. Amer. Math. Soc. 258 (1980), 127--146.

\bibitem[Mil56]{milnor:1956}
J.W. Milnor. Construction of universal bundles II. Ann. Math. 63
(1956), 430--436. 

\bibitem[MM92]{maclane:moerdijk:1992:sheaves}
S. Mac{ }Lane and I. Moerdijk. Sheaves in geometry and logic: a first
introduction to topos theory. Universitext, Springer (1992).

\bibitem[MV99]{morel:voevodsky:1999:a1}
F. Morel and V. Voevodsky. $\Ao$-homotopy theory of
schemes. Publ. Math. Inst. Hautes \'Etudes Sci. 90 (1999),
45--143. 

\bibitem[Pup74]{puppe:1974:hofib}
V. Puppe. A remark on homotopy fibrations. Manuscripta Math. 12 (1974),
113--120. 

\bibitem[Rez98]{rezk:1998:sharp}
C. Rezk. Fibrations and homotopy colimits of simplicial
sheaves. Preprint (1998), arXiv:math.AT/9811038.



\bibitem[Rud98]{rudyak:1998:cobordism}
Y.B. Rudyak. On Thom spectra, orientability and cobordism. Springer
Monographs in Mathematics, Springer (1998).

\bibitem[Rut97]{rutter:1997}
J.W. Rutter. Spaces of homotopy self-equivalences. A survey. Lecture
Notes in Mathematics, 1662. Springer-Verlag, Berlin, 1997.  

\bibitem[Sch82]{schoen:1982}
R. Sch\"on. The Brownian classification of fiber
spaces. Arch. Math. (Basel) 39 (1982),  359--365.  

\bibitem[Sta63]{stasheff:1963}
J.D. Stasheff. A classification theorem for fibre spaces. Topology 2
(1963), 239--246.

\bibitem[Sta70]{stasheff:1970}
J.D. Stasheff. $H$-spaces and classifying spaces: foundations and
recent developments. Algebraic topology (Proc. Sympos. Pure Math.,
Vol. XXII, Univ. Wisconsin, Madison, Wis., 1970),
pp. 247--272. Amer. Math. Soc., Providence, R.I., 1971.  

\bibitem[Sta74]{stasheff:1974}
J.D. Stasheff. Parallel transport and classification of
fibrations. Algebraic and geometrical methods in topology
(Conf. Topological Methods in Algebraic Topology, State Univ. New
York, Binghamton, N.Y., 1973), pp. 1--17. Lecture Notes in Math.,
Vol. 428, Springer, Berlin, 1974.  

\bibitem[SW06]{stasheff:wirth}
J.D. Stasheff and J. Wirth. Homotopy transition cocycles.
J. Homotopy Relat. Struct. 1 (2006), no. 1, 273--283. 

\bibitem[Wan80]{waner:1980:classify}
S. Waner. Equivariant classifying spaces and
fibrations. Trans. Amer. Math. Soc. 258 (1980), 385--405. 

\bibitem[Wen07]{thesis}
M. Wendt. On fibre sequences in motivic homotopy theory. PhD thesis,
Universit\"at Leipzig (2007). 

\bibitem[Wen09]{flocal}
M. Wendt. Fibre sequences and localization of simplicial
sheaves. Preprint, 2009.

\bibitem[Wen10]{toric}
M. Wendt. On the $\Ao$-fundamental groups of smooth toric
varieties. Adv. Math. 223 (2010), 352--378.

\end{thebibliography}
\end{document}